\documentclass{jocg}
\usepackage[english]{babel}

\title{%
  \MakeUppercase{Optimization by Gradient Boosting}%
}

\author{%
  G\'erard~Biau%
  \thanks{\affil{Sorbonne Universit\'es, UPMC, CNRS, and Institut Universitaire de France, Paris, France},
          \email{gerard.biau@upmc.fr}}    
       \,
and Beno\^{\i}t Cadre%
  \thanks{\affil{IRMAR, ENS Rennes,
CNRS, UBL, France},
          \email{benoit.cadre@ens-rennes.fr}}
}

\usepackage{float}
\usepackage{amsmath} 
\usepackage{amssymb}
\usepackage{amsfonts}
\usepackage{amsthm}
\usepackage{epstopdf}
\usepackage{color,fancybox,graphicx}
\usepackage{url}
\usepackage{psfrag}
\usepackage{color}
\usepackage{pifont}
\usepackage{mathrsfs}
\usepackage{dsfont}
\usepackage{enumerate}
\usepackage{pgf}
\usepackage{tikz}
\usetikzlibrary{arrows,automata}
\usepackage{natbib}
\usepackage{framed}
\usepackage[,slantedGreek]{mathptmx}

\usepackage{mathtools}

\usepackage{algorithm, algorithmic}
\floatname{algorithm}{Gradient Boosting Algorithm}

\bibpunct{(}{)}{;}{a}{,}{,}
\usepackage{hyperref}
\hypersetup{
colorlinks = true,
urlcolor = blue, 
linkcolor = blue,
citecolor = blue,
}

 {\endMakeFramed}

\newtheorem{rem}{Remark}[section]

\newtheorem{lem}{Lemma}[section]
\newtheorem{cor}{Corollary}[section]

\newtheorem{theo}{Theorem}[section]

\begin{document}
\maketitle

\begin{abstract}
Gradient boosting is a state-of-the-art prediction technique that sequentially produces a model in the form of linear combinations of simple predictors---typically decision trees---by solving an infinite-dimensional convex optimization problem. We provide in the present paper a thorough analysis of two widespread versions of gradient boosting, and introduce a general framework for studying these algorithms from the point of view of functional optimization. We prove their convergence as the number of iterations tends to infinity and highlight the importance of having a strongly convex risk functional to minimize. We  also present a reasonable statistical context ensuring consistency properties of the boosting predictors as the sample size grows. In our approach, the optimization procedures are run forever (that is, without resorting to an early stopping strategy), and statistical regularization is basically achieved via an appropriate $L^2$ penalization of the loss and strong convexity arguments.
\end{abstract}
\section{Introduction}
More than twenty years after the pioneering articles of Freund and Schapire \citep{Sc90,Fr95,FrSc96a,FrSc97}, boosting is still one of the most powerful ideas introduced in statistics and machine learning. Freund and Schapire's AdaBoost algorithm and its numerous descendants have proven to be competitive in a variety of applications, and are still able to provide state-of-the-art decisions in difficult real-life problems. In addition, boosting procedures are computationally fast and comfortable with both regression and classification problems.  For surveys of various aspects of boosting algorithms and related approaches, we refer to \citet{MeRa03}, \citet{BuHo07}, and to the monographs by \citet{HaTiFr09} and \citet{BuVd11}.

In a nutshell, the basic idea of boosting is to combine the outputs of many ``simple'' predictors, in order to produce a powerful committee with performances improved over the single members. Historically, the first formulations of Freund and Schapire considered boosting as an iterative classification algorithm that is run for a fixed number of iterations, and, at each iteration, selects one of the base classifiers, assigns a weight to it, and outputs the weighted majority vote of the chosen classifiers. Later on, \citet{Br97,Br98,Br99,Br00,Br04} made in a series of papers and technical reports the breakthrough observation that AdaBoost is in fact a gradient-descent-type algorithm in a function space, thereby identifying boosting at the frontier of numerical optimization and statistical estimation. This connection was further emphasized by \citet{FrHaTi00}, who rederived AdaBoost as a method for fitting an additive model in a forward stagewise manner. Following this, \citet{Fr01,Fr02} developed a general statistical framework (both for regression and classification) that $(i)$ yields a direct interpretation of boosting methods from the perspective of numerical optimization in a function space, and $(ii)$ generalizes them by allowing optimization of an arbitrary loss function. The term ``gradient boosting'' was coined by the author, who paid a special attention to the case where the individual additive components are decision trees. At the same time, \citet{MaBaBaFr99,MaBaBaFr00} embraced a more mathematical approach, revealing boosting as a principle to optimize a convex risk in a function space, by iteratively choosing a weak learner that approximately points in the negative gradient direction. 

This functional view of boosting has led to the development of algorithms in many areas of machine learning and computational statistics, beyond regression and classification. The history of boosting goes on today with algorithms such as XGBoost (Extreme Gradient Boosting, \citealp{ChGu16}), a tree boosting system widely recognized for its outstanding results in numerous data challenges. (An overview of its successes is given in the introductive section of the paper by \citealp{ChGu16}.) From a general point of view, XGBoost is but a scalable implementation of gradient boosting that contains various systems and algorithmic optimizations. Its mathematical principle is to perform a functional gradient-type descent in a space of decision trees, while regularizing the objective to avoid overfitting. 

However, despite a long list of successes, much work remains to be done to clarify the mathematical forces driving gradient boosting algorithms. Many influential articles regard boosting with a statistical eye and study the somewhat idealized problem of empirical risk minimization with a convex loss \citep[e.g.,][]{BlLuVa03,LuVa04}. These papers essentially concentrate on the statistical properties of the approach (that is, consistency and rates of convergence as the sample size grows) and often ignore the underlying optimization aspects. Other important articles, such as \citet{BuYu03,MaMeZh03,ZhYu05,BiRiZa06,BaTr07} take advantage of the iterative principle of boosting, but essentially focus on regularization via early stopping (that is, stopping the boosting iterations at some point), without paying too much attention to the optimization aspects. Notable exceptions are the pioneering notes of Breiman cited above, and the original paper by \citet{MaBaBaFr00}, who envision gradient boosting as an infinite-dimensional numerical optimization problem and pave the way for a more abstract analysis. All in all, there is to date no sound theory of gradient boosting in terms of numerical optimization. This state of affairs is a bit paradoxical, since optimization is certainly the most natural mathematical environment for gradient-descent-type algorithms. 

In line with the above, our main objective in this article is to provide a thorough analysis of two widespread models of gradient boosting, due to 
\citet{Fr01} and \citet{MaBaBaFr00}. We introduce in Section \ref{section2} a general framework for studying the algorithms from the point of view of functional optimization in an $L^2$ space, and prove in Section \ref{section3} their convergence  as the number of iterations tends to infinity. Our results allow for a large choice of convex losses in the optimization problem (differentiable or not), while highlighting the importance of having a strongly convex risk functional to minimize. This point is interesting, since it provides some theoretical justification for adding a penalty term to the objective, as advocated for example in the XGBoost system of \citealp{ChGu16}. Thus, the main message of Section \ref{section3} is that, under appropriate conditions, the sequence of boosted iterates converges towards the minimizer of the empirical risk functional over the set of linear combinations of weak learners. However, this does not guarantee that the output of the algorithms (i.e., the boosting predictor) enjoys good statistical properties, as overfitting may kick in. For this reason, we present in Section \ref{section4} a reasonable framework ensuring consistency properties of the boosting predictors as the sample size grows. In our approach, the optimization procedures are run forever (that is, without resorting to an early stopping strategy), and statistical regularization is basically achieved via an appropriate $L^2$ penalization of the loss and strong convexity arguments.

Before embarking on the analysis, we would like to stress that the present paper is theoretical in nature and that its main goal is to clarify/solidify some of the optimization ideas that are behind gradient boosting. In particular, we do not report experimental results, and refer to the specialized literature on (extreme) gradient boosting for discussions on the computational aspects and experiments with real-world data. 
\section{Gradient boosting}
\label{section2}
The purpose of this section is to describe the gradient boosting procedures that we analyze in the paper. 
\subsection{Mathematical context}
We assume to be given a sample $\mathscr D_n=\{(X_1,Y_1), \hdots, (X_n,Y_n)\}$ of i.i.d.~observations, where each pair $(X_i,Y_i)$ takes values in $\mathscr X \times \mathscr Y$. Throughout, $\mathscr X$ is a Borel subset of $\mathds R^d$, and $\mathscr Y\subset \mathds R$ is either a finite set of labels (for classification) or a subset of $\mathds R$ (for regression).  The vector space $\mathds R^d$ is endowed with the Euclidean norm $\|\cdot\|$. 

Our goal is to construct a predictor $F:\mathscr X\to \mathds R$ that assigns a response to each possible value of an independent random observation distributed as $X_1$.  In the context of gradient boosting, this general problem is addressed by considering a class $\mathscr F$ of functions $f:\mathscr X\to \mathds R$ (called the weak or base learners) and minimizing some empirical risk functional
\begin{equation*}
C_n(F)=\frac{1}{n}\sum_{i=1}^n\psi(F(X_i),Y_i)
\end{equation*}
over the linear combinations of functions in $\mathscr F$. The function $\psi:\mathds R \times \mathscr Y \to \mathds R_+$, called the loss, is convex in its first argument and measures the cost incurred by predicting $F(X_i)$ when the answer is $Y_i$. For example, in the least squares regression problem, $\psi(x,y)=(y-x)^2$ and
$$C_n(F)=\frac{1}{n}\sum_{i=1}^n(Y_i-F(X_i))^2.$$
However, many other examples are possible, as we will see below. Let $\delta_z$ denote the Dirac measure at $z$, and let $\mu_n=(1/n)\sum_{i=1}^n \delta_{(X_i,Y_i)}$ be the empirical measure associated with the sample $\mathscr D_n$. Clearly,
$$C_n(F)=\mathds E \psi(F(X),Y),$$
where $(X,Y)$ denotes a random pair with distribution $\mu_n$ and the symbol $\mathds E$ denotes the ex\-pec\-ta\-ti\-on with respect to $\mu_n$. Naturally, the theoretical (i.e., population) version of $C_n$ is 
\begin{equation*}
C(F)=\mathds E\psi(F(X_1),Y_1),
\end{equation*}
where now the expectation is taken with respect to the distribution of $(X_1,Y_1)$. It turns out that most of our subsequent developments are independent of the context, whether empirical or theoretical. Therefore, to unify the notation, we let throughout $(X,Y)$ be a generic pair of random variables with distribution $\mu_{X,Y}$, keeping in mind that $\mu_{X,Y}$ may be the distribution of $(X_1,Y_1)$ (theoretical risk), the standard empirical measure $\mu_n$ (empirical risk), or any smoothed version of $\mu_n$. 

We let $\mu_X$ be the distribution of $X$, $L^2(\mu_{X})$ the vector space of all measurable functions $f:\mathscr X\to \mathds R$ such that $\int |f|^2{\rm d}\mu_{X}<\infty$, and denote by $\langle \cdot,\cdot\rangle _{\mu_X}$ and $\|\cdot \|_{\mu_X}$ the corresponding norm and scalar product. Thus, for now, our problem is to minimize the quantity
$$
C(F)=\mathds E\psi(F(X),Y)
$$
over the linear combinations of functions in a given subset $\mathscr F$ of $L^2(\mu_X)$. A typical example for $\mathscr F$ is the collection of all binary decision trees in $\mathds R^d$ using axis parallel cuts with $k$ terminal nodes. In this case, each $fÊ\in \mathscr F$ takes the form $f=\sum_{j=1}^{k} \beta_j \mathds 1_{A_j}$, where 
$(\beta_1, \hdots, \beta_k) \in \mathds R^k$ and $A_1, \hdots, A_{k}$ is a tree-structured partition of $\mathds R^d$ \citep[][Chapter 20]{DeGyLu96}. 

As noted earlier, we assume that, for each $y \in \mathscr Y$, the function $\psi(\cdot, y)$ is convex. In the framework we have in mind, the function $\psi$ may take a variety of different forms, ranging from standard (regression or classification) losses to more involved penalized objectives. It may also be differentiable or not. Before discussing some examples in detail, we list assumptions that will be needed at some point. Throughout, we let $\xi(\cdot,y)$ be a subgradient of the convex function $\psi(\cdot, y)$, and recall that
 $\xi(x,y) \in [\partial_x^{-}\psi(x,y);\partial_x^{+}\psi(x,y)]$ (left and right partial derivatives).  In particular, for all $(x_1,x_2)\in \mathds R^2$,
\begin{equation}
\label{sub-zero}
\psi(x_1,y)\geq \psi(x_2,y)+\xi(x_2,y)(x_1-x_2).
\end{equation}

{\bf Assumption ${\bf A_1}$}
\begin{enumerate}
\item[${\bf A_1}$] One has $\mathds E \psi(0,Y)<\infty$. In addition, for all $F\in L^2(\mu_X)$, there exists $\delta>0$ such that
$$\sup_{G \in L^2(\mu_X):\|G-F\|_{\mu_X}\leq \delta}\big(\mathds E |\partial_x^{-}\psi(G(X),Y)|^2+\mathds E |\partial_x^{+}\psi(G(X),Y)|^2\big)<\infty.$$
\end{enumerate}

This assumption ensures that the convex functional $C$ is locally bounded (in particular, $C(F)<\infty$ for all $F\in L^2(\mu_X)$, and $C$ is continuous). Indeed, by inequality (\ref{sub-zero}), for all $G\in L^2(\mu_X)$,
\begin{align*}
\psi(G(x),y) &\leq \psi(0,y)+\xi(G(x),y)G(x).
\end{align*}
Therefore, using
$$|\xi(G(X),Y)|\leq |\partial_x^{-}\psi(G(X),Y)|+|\partial_x^{+}\psi(G(X),Y)|,$$
we have, by Assumption ${\bf A_1}$ and the Cauchy-Schwarz inequality,
$$\mathds E \psi(G(X),Y) \leq \mathds E\psi(0,Y)+\big(\mathds E\xi(G(X),Y)^2\mathds E G(X)^2\big)^{1/2},$$
so that $C$ is locally bounded. Naturally, Assumption ${\bf A_1}$ is automatically satisfied for the choice $\mu_{X,Y}=\mu_n$. 

{\bf Assumption ${\bf A_2}$}
\begin{enumerate}
\item[${\bf A_2}$] There exists $\alpha>0$ such that, for all $y\in \mathscr Y$, the function $\psi(\cdot,y)$ is $\alpha$-strongly convex, i.e., for all $(x_1,x_2)\in \mathds R^2$ and $t\in [0,1]$,
$$\psi(tx_1+(1-t)x_2,y)\leq t\psi(x_1,y)+(1-t)\psi(x_2,y)-\frac{\alpha}{2} t(1-t)(x_1-x_2)^2.$$
\end{enumerate}

This assumption will be used in most, but not all, results. Strong convexity will play an essential role in the statistical Section \ref{section4}. We note that, under Assumption ${\bf A_2}$, for all $(x_1,x_2)\in \mathds R^2$,
\begin{equation}
\label{sub-one}
\psi(x_1,y)\geq \psi(x_2,y)+\xi(x_2,y)(x_1-x_2)+\frac{\alpha}{2}(x_1-x_2)^2,
\end{equation}
which is of course an inequality tighter than (\ref{sub-zero}). Furthermore, the $\alpha$-strong convexity of $\psi(\cdot,y)$ implies the $\alpha$-strong convexity of the risk functional $C$ over $L^2(\mu_X)$. 

In addition to Assumptions ${\bf A_1}$ and ${\bf A_2}$, we require the following:
\begin{enumerate}
\item[${\bf A_3}$] There exists a positive constant $L$ such that, for all $(x_1,x_2) \in \mathds R^2$,
$$|\mathds E( \xi(x_1,Y)-\xi(x_2,Y)\,|\,X)| \leq L |x_1-x_2|.$$
\end{enumerate}

However esoteric this assumption may seem, it is in fact mild and provide a framework that encompasses a large variety of familiar situations. In particular, Assumption ${\bf A_3}$ admits a stronger version ${\bf A'_3}$, which is useful as soon as the function $\psi$ is continuously differentiable with respect to its first variable:

\begin{enumerate}
\item[${\bf A'_3}$] For all $y\in \mathscr Y$, the function $\psi(\cdot,y)$ is continuously differentiable, and there exists a positive constant $L$ such that, for all $(x_1,x_2,y) \in \mathds R^2\times \mathscr Y$,
$$|\partial _x \psi(x_1,y)-\partial _x \psi(x_2,y)| \leq L |x_1-x_2|.$$
\end{enumerate}

Assumption ${\bf A'_3}$ implies ${\bf A_3}$, but the converse is not true. To see this, just note that, in the smooth situation ${\bf A'_3}$, we have $\xi(x,y)=\partial_x \psi(x,y)$. Therefore,
$$\mathds E (\xi(x_1,Y)\,|\,X)=\int \partial_x \psi(x_1,Y)\mu_{Y|X}({\rm d}y),$$
where $\mu_{Y|X}$ is the conditional distribution of $Y$ given $X$. We also note that, in the context of ${\bf A'_3}$ , the functional $C$ is differentiable at any $F\in L^2(\mu_X)$ in the direction $G\in L^2(\mu_X)$, with differential
$$dC(F;G)=\langle\nabla C(F),G \rangle_{\mu_X},$$
where $\nabla C(F)(x):=\int \partial_x \psi(F(x),y)\mu_{Y|X=x}({\rm d}y)$ is the gradient of $C$ at $F$. However, Assumption ${\bf A_3}$ allows to deal with a larger variety of losses, including non-differentiable ones, as shown in the examples below.
\subsection{Some examples}
\begin{itemize}
\item  A first canonical example, in the regression setting, is to let $\psi(x,y)=(y-x)^2$ (squa\-red error loss), which is $2$-strongly convex in its first argument (Assumption $\bf A_2$) and satisfies Assumption ${\bf A_1}$ as soon as $\mathds E Y^2<\infty$. It also satisfies  ${\bf A'_3}$, with $\partial_x \psi(x,y)=2(x-y)$ and $L=2$.
\item Another example in regression is the loss $\psi(x,y)=|y-x|$ (absolute error loss), which is convex but not strongly convex in its first argument. Whenever strong convexity of the loss is required, a possible strategy is to regularize the objective via an $L^2$-type penalty, and take
$$\psi(x,y)=|y-x|+\gamma x^2,$$
where $\gamma$ is a positive parameter (possibly function of the sample size $n$ in the empirical setting). This loss is $(2\gamma)$-strongly convex in $x$ and satisfies ${\bf A_1}$ and ${\bf A_2}$ whenever $\mathds E |Y|<\infty$,  with $\xi(x,y)=\mbox{sgn}(x-y)+2\gamma x$ (with $\mbox{sgn}(u)=2\mathds 1_{[u> 0]}-1$ for $u\neq 0$ and $\mbox{sgn}(0)=0$). On the other hand, the function $\psi(\cdot,y)$ is not differentiable at $y$, so that the smoothness Assumption $\bf A'_3$ is not satisfied. However,
\begin{align*}
\mathds E( \xi(x_1,Y)-\xi(x_2,Y)\,|\,X) &=\int (\mbox{sgn}(x_1-y)-\mbox{sgn}(x_2-y))\mu_{Y|X}({\rm d}y)+2\gamma(x_1-x_2)\\
&=\mu_{Y|X}((-\infty,x_1))-\mu_{Y|X}((-\infty,x_2))+2\gamma(x_1-x_2)\\
& \quad -\mu_{Y|X}((x_1,\infty))+\mu_{Y|X}((x_2,\infty)).
\end{align*}
Thus, if we assume for example that $\mu_{Y|X}$ has a density (with respect to the Lebesgue measure) bounded by $B$, then
$$|\mathds E( \xi(x_1,Y)-\xi(x_2,Y)\,|\,X)| \leq 2(B+\gamma) |x_1-x_2|,$$
and Assumption ${\bf A_3}$ is therefore satisfied. Of course, in the empirical setting, assuming that $\mu_{Y|X}$ has a density precludes the use of the empirical measure $\mu_n$ for $\mu_{X,Y}$. A safe and simple alternative is to consider a smoothed version $\tilde \mu_n$ of $\mu_n$ \citep[based, for example, on a kernel estimate; see][]{DeGy85}, and to minimize the functional
$$C_n(F)=\int |y-F(x)|\tilde\mu_n({\rm d}x,{\rm d}y)+\gamma \int F(x)^2 \tilde\mu_n({\rm d}x)$$
over the linear combinations of functions in $\mathscr F$. 
\item In the $\pm1$-classification problem, the final classification rule is $+1$ if $F(x)>0$ and $-1$ otherwise. Often, the function $\psi(x,y)$ has the form $\phi (yx)$, where $\phi :\mathds R\to \mathds R_+$ is convex. Classical losses include the choices $\phi(u)=\ln_2(1+e^{-u})$ (logit loss), $\phi(u)=e^{-u}$ (exponential loss), and $\phi(u)=\max(1-u,0)$ (hinge loss). None of these losses is strongly convex, but here again, this can be repaired whenever needed by regularizing the problem via
\begin{equation}
\label{L+P}
\psi(x,y)=\phi(yx)+\gamma x^2,
\end{equation}
where $\gamma> 0$. It is for example easy to see that $\psi(x,y)=\ln_2(1+e^{-yx})+\gamma x^2$ satisfies Assumptions ${\bf A_1}$, ${\bf A_2}$, and ${\bf A'_3}$. This is also true for the penalized sigmoid loss $\psi(x,y)=(1-\tanh(\beta yx))+\gamma x^2$, where $\beta$ is a positive parameter. In this case, $\psi(\cdot,y)$ is $2(\gamma-\beta^2)$-strongly convex as soon as $\beta<\sqrt{\gamma}$. Another interesting example in the classification setting is the loss $\psi(x,y)=\phi(xy)+\gamma x^2$, where
\begin{equation*} 
\phi(u) = \left\{
\begin{array}{ll}
-u+1 & \mbox{if $u\leq 0$}\\
e^{-u} & \mbox{if $u> 0$.}
\end{array}
\right.
\end{equation*}
We leave it as an easy exercise to prove that Assumptions ${\bf A_1}$, ${\bf A_2}$, and ${\bf A'_3}$ are satisfied. Examples could be multiplied endlessly, but the point we wish to make is that our assumptions are mild and allow to consider a large variety of learning problems. We also emphasize that regularized objectives of the form (\ref{L+P}) are typically in action in the Extreme Gradient Boosting system of \citet{ChGu16}. 
\end{itemize}
\subsection{Two algorithms}
Let ${\mbox{lin}(\mathscr F)}$ be the set of all linear combinations of functions in $\mathscr F$, our collection of base predictors in $L^2(\mu_X)$. So, each $F \in {\mbox{lin}(\mathscr F)}$ has the form $F=\sum_{j=1}^J \beta_j f_j$, where $(\beta_1, \hdots, \beta_J) \in \mathds R^J$ and $f_1, \hdots, f_J$ are elements of $\mathscr F$. Finding the infimum of the functional $C$ over ${\mbox{lin}(\mathscr F)}$ is a challenging infinite-dimensi\-o\-nal optimization problem, which requires an algorithm. The core idea of the gradient boosting approach is to greedily locate the infimum by producing a combination of base predictors via a gradient-descent-type algorithm in $L^2(\mu_X)$. Focusing on the basics, this can be achieved by two related yet different strategies, which we examine in greater mathematical details below. \hyperref[algorithm1]{Algorithm 1} appears in \citet{MaBaBaFr00}, whereas \hyperref[algorithm2]{Algorithm 2} is essentially due to \citet{Fr01}. 

It is implicitly assumed throughout this paragraph that Assumption ${\bf A_1}$ is satisfied. We recall that under this assumption, the convex functional $C$ is locally bounded and therefore continuous. Thus, in particular,
$$\inf_{F \in \mbox{\footnotesize lin}(\mathscr F)} C(F)=\inf_{F \in \overline{\mbox{\footnotesize lin}(\mathscr F)}}C(F),$$
where $\overline{\mbox{lin}(\mathscr F)}$ is the closure of $\mbox{lin}(\mathscr F)$ in $L^2(\mu_X)$.
Loosely speaking, looking for the infimum of $C$ over $\overline{\mbox{lin}(\mathscr F)}$ is the same as looking for the infimum of $C$ over all (finite) linear combinations of ``small'' functions in $\mathscr F$. We note in addition that if Assumption ${\bf A_2}$ is satisfied, then there exists a unique function $\bar F \in \overline{\mbox{lin}(\mathscr F)}$ (which we call the boosting predictor) such that 
\begin{equation}
\label{Fbar}
C(\bar F)=\inf_{F \in {\mbox{\footnotesize lin}}(\mathscr F)}C(F).
\end{equation}
\paragraph{\hyperref[algorithm1]{Algorithm 1}.} In this approach, we consider a class $\mathscr F$ of functions $f:\mathscr X\to \mathds R$ such that $0 \in \mathscr F$, $f\in \mathscr F \Leftrightarrow -f \in \mathscr F$, and $\|f\|_{\mu_X}=1$ for $f\neq 0$. An example is the collection $\mathscr F$ of all $\pm 1$-binary trees in $\mathds R^d$ using axis parallel cuts with $k$ terminal nodes (plus zero). Each nonzero $fÊ\in \mathscr F$ takes the form $f=\sum_{j=1}^{k} \beta_j \mathds 1_{A_j}$, where 
$|\beta_j|=1$ and $A_1, \hdots, A_{k}$ is a tree-structured partition of $\mathds R^d$ \citep[][Chapter 20]{DeGyLu96}. 
The parameter $k$ is a measure of the tree complexity. For example, trees with $k=d+1$ are such that $\overline{\mbox{lin}(\mathscr F)}=L^2(\mu_X)$ \citep{Br00}. Thus, in this case,
$$\inf_{F \in {\mbox{\footnotesize lin}}(\mathscr F)}C(F)=\inf_{F \in L^2(\mu_X)} C(F).$$
Although interesting from the point of view of numerical optimization, this situation is however of little interest for statistical learning, as we will see in Section \ref{section4}.

Suppose now that we have a function $F \in \mbox{lin}(\mathscr F)$ and wish to find a new $f\in \mathscr F$ to add to $F$ so that the risk $C(F+wf)$ decreases at most, for some small value of $w$. Viewed in function space terms, we are looking for the direction $f \in \mathscr F$ such that $C(F+wf)$ most rapidly decreases. Assume for the moment, to simplify, that $\psi$ is continuously differentiable in its first argument. Then the knee-jerk reaction is to take the opposite of the gradient of $C$ at $F$, but since we are restricted to choosing our new function in $\mathscr F$, this will in general not be a possible choice. Thus, instead, we start from the approximate identity
\begin{equation}
\label{GD1}
C(F)-C(F+wf) \approx -w \langle \nabla C (F),f\rangle_{\mu_X}
\end{equation}
and choose $f \in \mathscr F$ that maximizes $-\langle \nabla C (F),f\rangle_{\mu_X}$. For an arbitrary (i.e., not necessarily differentiable) $\psi$, we simply replace the gradient by a subgradient and choose $f \in \mathscr F$  that maximizes $-\mathds E\xi(F(X),Y)f(X)$. This motivates the following iterative algorithm:
\begin{algorithm}
\caption{}
\label{algorithm1}
\begin{algorithmic}[1]
\STATE {\bf Require} $(w_{t})_{t}$ a sequence of positive real numbers.
\STATE {\bf Set} $t=0$ and start with $F_{0} \in \mathscr F$.
\STATE {\bf Compute}
\begin{equation}
\label{optimizationstep}
f_{t+1}\in {\arg \max}_{f \in \mathscr F} -\mathds E \xi (F_t(X),Y)f(X)
\end{equation}
and {\bf let} $F_{t+1}=F_{t}+w_{t+1} f_{t+1}$.
\STATE {\bf Take} $t\leftarrow t+1$ and {\bf go} to step 3.
\end{algorithmic}
\end{algorithm}

We emphasize that the method performs a gradient-type descent in the function space $L^2(\mu_X)$, by choosing at 
each iteration a base predictor to include in the combination so as to maximally reduce the value of the risk functional. However, the main difference with a standard gradient descent is that \hyperref[algorithm1]{Algorithm 1} forces the descent direction to belong to $\mathscr F$. To understand the rationale behind this principle, assume that $\psi$ is continuously differentiable in its first argument. As we have seen earlier, in this case,
$$-\mathds E \xi (F_t(X),Y)f(X) =-\langle \nabla C (F_t),f\rangle_{\mu_X},$$
and, for $\nabla C (F_t)\neq 0$, 
$$\frac{-\nabla C (F_t)}{\|\nabla C (F_t)\|_{\mu_X}}={\arg\max}_{F\in L^2(\mu_X):\|F\|_{\mu_X}=1}-\langle \nabla C (F_t),F\rangle_{\mu_X}.$$
Thus, at each step, \hyperref[algorithm1]{Algorithm 1} mimics the computation of the negative gradient by restricting the search of the supremum to the class $\mathscr F$, i.e., by taking
$$f_{t+1}\in {\arg\max}_{f\in \mathscr F}-\langle \nabla C (F_t),f\rangle_{\mu_X},$$
which is exactly (\ref{optimizationstep}). In the empirical case (i.e., $\mu_{X,Y}=\mu_n$) this descent step takes the form
$$f_{t+1}\in {\arg \max}_{f \in \mathscr F} -\frac{1}{n}\sum_{i=1}^n \nabla C (F_t)(X_i)\cdot f(X_i).$$
Finding this optimum is a non-trivial computational problem, which necessitates a strategy. For example, in the spirit of the CART algorithm of \citet{BrFrOlSt84}, \citet{ChGu16} use in the XGBoost package a greedy approach that starts from a single leaf and iteratively adds branches to the tree.

The sequence $(w_t)_{t}$ is the sequence of
step sizes, which are allowed to change at every iteration and should be carefully chosen for convergence
guarantees. It is also stressed that the algorithm is assumed to be run forever, i.e., 
stopping or not the iterations is not an issue at this stage of the analysis. As we will see in the next section, the algorithm is convergent under our assumptions (with an appropriate choice of the sequence $(w_t)_{t}$), in the sense that 
$$\lim_{t\to \infty} C(F_t)=\inf_{F\in \mbox{\footnotesize lin}(\mathscr F)}C(F).$$
Of course, in the empirical case, the statistical properties as $n \to \infty$ of the limit deserve a special treatment, connected with possible overfitting issues. This important discussion is postponed to Section \ref{section4}.
\paragraph{\hyperref[algorithm2]{Algorithm 2}.} The principle we used so far rests upon the simple Taylor-like identity (\ref{GD1}), which  encourages us to imitate the definition of the negative gradient in the class $\mathscr F$. Still starting from (\ref{GD1}), there is however another strategy, maybe more natural, which consists in choosing $f_{t+1}$ by a least squares approximation of $-\xi(F_t(X),Y)$. To follow this route, we modify a bit the collection of weak learners, and consider a class $\mathscr P\subset L^2(\mu_X)$ of functions $f:\mathscr X\to \mathds R$ such that $f\in \mathscr P \Leftrightarrow -f \in \mathscr P$, and $af \in \mathscr P$ for all $(a,f)\in \mathds R\times \mathscr P$ (in particular, $0 \in \mathscr P$, which is thus a cone of $L^2(\mu_X)$).  Binary trees in $\mathds R^d$ using axis parallel cuts with $k$ terminal nodes are a good example of a possible class $\mathscr P$. These base learners are of the form $f=\sum_{j=1}^{k} \beta_j \mathds 1_{A_j}$, where this time $(\beta_1, \hdots, \beta_{k})\in \mathds R^{k}$, without any normative constraint.

Given $F_t$, the idea of \hyperref[algorithm2]{Algorithm 2} is to choose $f_{t+1} \in \mathscr P$ that minimizes the squared norm between $-\xi(F_t(X),Y)$ and $f_{t+1}(X)$, i.e., to let
$$f_{t+1}\in {\arg\min}_{f \in \mathscr P}\mathds E (-\xi(F_{t}(X),Y)-f(X))^2,$$
or, equivalently, 
$$f_{t+1}\in {\arg\min}_{f \in \mathscr P} \big(2\mathds E \xi(F_{t}(X),Y)f(X)+\|f\|_{\mu_X}^2\big).$$
A more algorithmic format is shown below.
\begin{algorithm}
\caption{}
\label{algorithm2}
\begin{algorithmic}[1]
\STATE {\bf Require} $\nu$ a positive real number.
\STATE {\bf Set} $t=0$ and start with $F_{0} \in \mathscr P$.
\STATE {\bf Compute}
\begin{equation}
\label{optimizationstep2}
f_{t+1}\in {\arg\min}_{f \in \mathscr P} \big(2\mathds E \xi(F_{t}(X),Y)f(X)+\|f\|_{\mu_X}^2\big)
\end{equation}
and {\bf let} $F_{t+1}=F_{t}+\nu f_{t+1}$.
\STATE {\bf Take} $t\leftarrow t+1$ and {\bf go} to step 3.
\end{algorithmic}
\end{algorithm}
We note that, contrary to \hyperref[algorithm1]{Algorithm 1}, the step size $\nu$ is kept fixed during the iterations. We will see
in the next section that choosing a small enough $\nu$ (depending in particular on the Lipschitz constant of Assumption ${\bf A_3}$) is sufficient to ensure the convergence of the algorithm. In the empirical setting, assuming that $\psi$ is continuously differentiable in its first argument, the optimization step (\ref{optimizationstep2}) reads
$$f_{t+1}\in {\arg \max}_{f \in \mathscr P} \frac{1}{n}\sum_{i=1}^n (-\nabla C (F_t)(X_i)-f(X_i))^2.$$
Therefore, in this context, the gradient boosting algorithm fits $f_{t+1}$ to the negative gradient instances $-\nabla C (F_t)(X_i)$ via a least squares minimization. When $\psi(x,y)=(y-x)^2/2$, then $-\nabla C (F_t)(X_i)=Y_i-F_t(X_i)$, and the algorithm simply fits $f_{t+1}$ to the residuals $Y_i-F_t(X_i)$ at step $t$, in the spirit of original boosting procedures. This observation is at the source of gradient boosting, which \hyperref[algorithm2]{Algorithm 2} generalizes to a much larger variety of loss functions and to more abstract measures.
\section{Convergence of the algorithms}
\label{section3}
This section is devoted to analyzing the convergence of the gradient boosting \hyperref[algorithm1]{Algorithm 1} and \hyperref[algorithm2]{Algorithm 2} as the number of iterations $t$ tends to infinity. Despite its importance, no results (or only partial answers) have been reported so far on this question.
\subsection{\hyperref[algorithm1]{Algorithm 1}}
The convergence of this algorithm rests upon the choice of the step size sequence $(w_t)_t$, which needs to be carefully specified. We take $w_0>0$ arbitrarily and set
\begin{equation}
\label{choixW}
w_{t+1}=\min \big(w_t,-(2L)^{-1}\mathds E \xi(F_t(X),Y)f_{t+1}(X)\big), \quad t\geq 0,
\end{equation}
where $L$ is the Lipschitz constant of Assumption ${\bf A_3}$. Clearly, the sequence $(w_t)_t$ is nonincreasing. It is also nonnegative. To see this, just note that, by definition,
$$f_{t+1}\in {\arg\max}_{f \in \mathscr F}-\mathds E \xi(F_t(X),Y)f(X),$$
and thus, since $0\in \mathscr F$, $-\mathds E \xi(F_t(X),Y)f_{t+1}(X)\geq 0$.
The main result of this section is encapsulated in the following theorem.
\begin{theo}
\label{convergenceCn}
Assume that Assumptions ${\bf A_1}$ and ${\bf A_3}$ are satisfied, and let $(F_t)_t$ be defined by \hyperref[algorithm1]{Algorithm 1} with $(w_t)_t$ as in (\ref{choixW}). Then
$$\lim_{t \to \infty}C(F_t)=\inf_{F \in {\emph{\footnotesize lin}(\mathscr F)}}C(F).$$
\end{theo}

Observe that Theorem \ref{convergenceCn} holds without Assumption ${\bf A_2}$, i.e., there is no need here to assume that the function $\psi(x,y)$ is strongly convex in $x$. However, whenever Assumption ${\bf A_2}$ is satisfied, there exists as in (\ref{Fbar}) a unique boosting predictor $\bar F \in \overline{\mbox{lin}(\mathscr F)}$ such that 
$C(\bar F)=\inf_{F \in {\mbox{\footnotesize lin}(\mathscr F)}}C(F)$, and the theorem guarantees that $\lim_{t \to \infty}C(F_t)=C(\bar F)$.

The proof of the theorem relies on the following lemma, which states that the sequence $(C(F_t))_t$ is nonincreasing. Since $C(F)$ is nonnegative for all $F$, we conclude that $C(F_t)\downarrow \inf_{k} C(F_k)$ as $t\to \infty$. This is the key argument to prove the convergence of $C(F_t)$ towards $\inf_{F \in {\mbox{\footnotesize lin}(\mathscr F)}}C(F)$.
\begin{lem}
\label{Cndecreases}
Assume that Assumptions ${\bf A_1}$ and ${\bf A_3}$ are satisfied. Then, for each $t\geq 0$,
$$C(F_t)-C(F_{t+1}) \geq L w^2_{t+1}.$$
In particular, $C(F_t) \downarrow \inf_k C(F_k)$ as $t\to \infty$, $\sum_{t\geq 1}w^2_{t}<\infty$, and $\lim_{t\to \infty}w_t = 0$. 
\end{lem}
\begin{proof}
Let $t\geq 0$. Recall that $F_{t+1}=F_t+w_{t+1}f_{t+1}$. If $f_{t+1}=0$, then $w_{t+1}=0$ and $F_{t+1}=F_{t}$, so that there is nothing to prove. Thus, in the remainder of the proof, it is assumed that $f_{t+1}$ is different from zero and, in turn, that $\|f_{t+1}\|_{\mu_X}=1$. 
Applying technical Lemma \ref{technical-1}, we may write
\begin{align*}
C(F_t) &\geq C(F_{t+1})-w_{t+1}^2L-w_{t+1}\mathds E \xi(F_t(X),Y)f_{t+1}(X)\\
& \geq C(F_{t+1})-w_{t+1}^2L+2L w_{t+1}\min\big( w_t,-(2L)^{-1}\mathds E\xi(F_{t}(X),Y)f_{t+1}(X)\big)\\
&=C(F_{t+1})+Lw_{t+1}^2,
\end{align*}
by definition (\ref{choixW}) of the sequence $(w_t)_t$.
\end{proof}
\begin{proof}[Proof of Theorem \ref{convergenceCn}]
Assume that, for some $t_0\geq 0$, $\sup_{f \in \mathscr F}-\mathds E\xi(F_{t_0}(X),Y)f(X)=0$. Then, by the symmetry of the class $\mathscr F$, for all $f\in \mathscr F$, $\mathds E\xi(F_{t_0}(X),Y)f(X)=0$. We conclude by technical Lemma \ref{technical-2} that
$$C(F_{t})=\inf_{F \in {\mbox{\footnotesize lin}(\mathscr F)}}C(F) \quad\mbox{for all } t\geq t_0,$$
and the result is proved. Thus, in the following, it is assumed that 
$$\sup_{f \in \mathscr F}-\mathds E \xi(F_t(X),Y)f(X)>0 \quad \mbox{for all }t\geq 0.$$
Consequently, $-\mathds E \xi(F_t(X),Y)f_{t+1}(X)>0$ and $w_t>0$ for all $t$. Since $w_t\to 0$, there exists a subsequence $(w_{t'})_{t'}$ such that
\begin{align}
w_{t'+1}&=-(2L)^{-1}\mathds E \xi(F_{t'}(X),Y)f_{t'+1}(X) \nonumber\\
&=(2L)^{-1}\sup_{f \in \mathscr F}-\mathds E \xi(F_{t'}(X),Y)f(X).\label{140717}
\end{align}
Let $\varepsilon >0$. For all $t'$ large enough and all $f\in \mathscr F$, by the symmetry of $\mathscr F$,
$$-\mathds E \xi(F_{t'}(X),Y)f(X) \leq \varepsilon \quad \mbox{and} \quad \mathds E \xi(F_{t'}(X),Y)f(X)\leq \varepsilon,$$ 
and thus $\lim_{t'\to \infty}\mathds E \xi(F_{t'}(X),Y)f(X)=0$ for all $f\in \mathscr F$. We conclude that, for all $G\in \mbox{lin}(\mathscr F)$,
\begin{equation}
\label{LS0}
\lim_{t'\to \infty}\mathds E \xi(F_{t'}(X),Y)G(X) = 0. 
\end{equation}

Assume, without loss of generality, that $F_0=0$, and observe that $F_t=\sum_{k=1}^t w_k f_k$. Thus, we may write 
\begin{align*}
\mathds E \xi(F_{t'}(X),Y)F_{t'}(X)&=\sum_{k=1}^{t'}w_k \mathds E \xi(F_{t'}(X),Y)f_k(X)\\
&\leq \sup_{f \in \mathscr F} \mathds E \xi(F_{t'}(X),Y)f(X)\sum_{k=1}^{t'}w_k\\
&=\sup_{f \in \mathscr F}- \mathds E \xi(F_{t'}(X),Y)f(X)\sum_{k=1}^{t'}w_k\\
& \quad \mbox{(by the symmetry of $\mathscr F$)}\\
& =2Lw_{t'+1}\sum_{k=1}^{t'}w_k,
\end{align*}
by definition of $w_{t'+1}$---see (\ref{140717}). So,
\begin{align*}
\mathds E \xi(F_{t'}(X),Y)F_{t'}(X) \leq 2Lw_{t'}\sum_{k=1}^{t'}w_k &=2Lw_{t'}\sum_{k=1}^{t'}w_k^{-1}w_k^2\\
&\mbox{(because $w_{t'+1}\leq w_{t'}$)}.
\end{align*}
Since $\sum_{k\geq 1} w_k^2<\infty$, and since the sequence $(w_t)_t$ is nonincreasing, positive, and tends to 0 as $t\to \infty$, Kronecker's lemma reveals that $w_{t'}\sum_{k=1}^{t'}w_k^{-1}w_k^2 \to 0$ as $t'\to \infty$. Therefore,
\begin{equation}
\label{LS1}
\limsup_{t'\to \infty} \mathds E \xi(F_{t'}(X),Y)F_{t'}(X)\leq 0.
\end{equation}

Let $\varepsilon>0$ and let $F_{\varepsilon}^{\star} \in \mbox{lin}(\mathscr F)$ be such that
$$\inf_{F \in {\mbox{\footnotesize lin}(\mathscr F)}}C(F) \geq C(F^{\star}_{\varepsilon})-\varepsilon.$$
By the convexity of $C$, we have, for all $t'$,
\begin{align*}
\inf_{F \in {\mbox{\footnotesize lin}(\mathscr F)}}C(F)&\geq C(F^{\star}_{\varepsilon}) - \varepsilon \\
& \geq C(F_{t'})+ \mathds E \xi(F_{t'}(X),Y)(F_{\varepsilon}^{\star}(X)-F_{t'}(X))-\varepsilon\\
&\geq \inf_k C(F_k) +\mathds E \xi(F_{t'}(X),Y)F_{\varepsilon}^{\star}(X)-\mathds E \xi(F_{t'}(X),Y)F_{t'}(X)-\varepsilon.
\end{align*}
Combining (\ref{LS0}) and (\ref{LS1}), we conclude that $\inf_{F \in {\mbox{\footnotesize lin}(\mathscr F)}}C(F) \geq \inf_kC(F_k)-\varepsilon$ for all $\varepsilon>0$, so that
$$\lim_{t \to \infty} C(F_t)=\inf_{k} C(F_k) = \inf_{F \in {\mbox{\footnotesize lin}(\mathscr F)}}C(F),$$ 
which is the desired result.
 \end{proof}

Theorem \ref{convergenceCn} ensures that the risk of the boosting iterates gets closer and closer to the minimal risk as the number of iterations grows. It turns out that, whenever $\overline{\mbox{lin}(\mathscr F)}=L^2(\mu_X)$, under Assumption ${\bf A_2}$ and the smooth framework of Assumption ${\bf A'_3}$, the sequence $(F_t)_t$ itself approaches $\bar F={\arg\min}_{F \in L^2(\mu_X)}C(F)$, as shown in Corollary \ref{coro-convergenceCn} below. This corollary is an easy consequence of Theorem \ref{convergenceCn} and the strong convexity of $C$. 
\begin{cor}
\label{coro-convergenceCn}
Assume that $\overline{\emph{lin}(\mathscr F)}=L^2(\mu_X)$. Assume, in addition, that Assumptions ${\bf A_1}$, ${\bf A_2}$, and ${\bf A'_3}$ are satisfied, and let $(F_t)_t$ be defined by \hyperref[algorithm1]{Algorithm 1} with $(w_t)_t$ as in (\ref{choixW}). Then
$$\lim_{t \to \infty}\|F_t-\bar F\|_{\mu_X}=0,$$
where
$$
\bar F = {\arg\min}_{F \in L^2(\mu_X)} C(F).
$$
\end{cor}
\begin{proof}
By the $\alpha$-strong convexity of $C$, 
$$C(F_t) \geq C(\bar F) +\mathds E \xi(\bar F,Y)(F_t-\bar F)+ \frac{\alpha}{2}\|F_t-\bar F\|_{\mu_X}^2,$$
which, under ${\bf A'_3}$, takes the more familiar form
$$C(F_t) \geq C(\bar F) +\langle \nabla C(\bar F),F_t-\bar F\rangle_{\mu_X}+ \frac{\alpha}{2}\|F_t-\bar F\|_{\mu_X}^2.$$
But, since $\bar F = {\arg\min}_{F \in L^2(\mu_X)} C(F)$, we know that $\langle \nabla C(\bar F),F_t-\bar F\rangle_{\mu_X}=0$. Thus,
$$C(F_t) -C(\bar F) \geq  \frac{\alpha}{2}\|F_t-\bar F\|_{\mu_X}^2,$$
and the conclusion follows from Theorem \ref{convergenceCn}.
 \end{proof}
\subsection{\hyperref[algorithm2]{Algorithm 2}} Recall that, in this context, each iteration picks an $f_{t+1} \in \mathscr P$ that satisfies
$$2\mathds E \xi(F_{t}(X),Y)f_{t+1}(X)+\|f_{t+1}\|_{\mu_X}^2\leq 2\mathds E \xi(F_{t}(X),Y)f(X)+\|f\|_{\mu_X}^2 \quad \mbox{for all } f\in \mathscr P.$$
\begin{theo}
\label{convergenceCn-algo2} 
Assume that Assumptions ${\bf A_1}$-${\bf A_3}$ are satisfied, and let $(F_t)_t$ be defined by \hyperref[algorithm2]{Algorithm 2} with  $0<\nu<1/(2L)$. Then
$$\lim_{t \to \infty}C(F_t)=\inf_{F \in {\emph{\footnotesize lin}(\mathscr P)}}C(F).$$
\end{theo}

The architecture of the proof is similar to that of Theorem \ref{convergenceCn}. (Note however that this theorem requires the strong convexity Assumption ${\bf A_2}$). In particular, we need the following important lemma, which states that the risk of the iterates decreases at each step of the algorithm.
\begin{lem}
\label{Cndecreases-algo2}
Assume that Assumptions ${\bf A_1}$ and ${\bf A_3}$ are satisfied, and let $0<\nu<1/(2L)$. Then, for each $t\geq 0$,
$$C(F_t)-C(F_{t+1}) \geq \frac{\nu}{2}(1-2\nu L) \|f_{t+1}\|_{\mu_X} ^2.$$
In particular, $C(F_t) \downarrow \inf_k C_k$ as $t\to \infty$, $\sum_{t\geq 1}\|f_t\|^2_{\mu_X}<\infty$, and $\lim_{t\to \infty}\|f_t\|_{\mu_X} = 0$. 
\end{lem}
\begin{proof}
Let $t\geq 0$. Applying technical Lemma \ref{technical-1}, we may write
\begin{align*}
C(F_t)&\geq C(F_{t+1})-\nu^2L\|f_{t+1}\|_{\mu_X}^2-\nu\mathds E \xi(F_t(X),Y)f_{t+1}(X)\\
&=C(F_{t+1})-\nu^2L\|f_{t+1}\|_{\mu_X}^2-\frac{\nu}{2}\big(2\mathds E \xi(F_t(X),Y)f_{t+1}(X)+\|f_{t+1}\|_{\mu_X}^2\big)+\frac{\nu}{2}\|f_{t+1}\|_{\mu_X}^2.
\end{align*}
Upon noting that $2\mathds E \xi(F_t(X),Y)f_{t+1}(X)+\|f_{t+1}\|_{\mu_X}^2\leq 0$ (since $0\in \mathscr P$), we conclude that
$$C(F_t)\geq C(F_{t+1})+\frac{\nu}{2}(1-2\nu L) \|f_{t+1}\|_{\mu_X}^2.$$
\end{proof}
\begin{proof}[Proof of Theorem \ref{Cndecreases-algo2}]
The first step is to establish that there exists a subsequence $(F_{t'})_{t'}$ such that $\lim_{t'\to \infty}\mathds E \xi(F_{t'}(X),Y)G(X)\to 0$ for all $G\in \mbox{lin}(\mathscr P)$. We start by observing that $C(F_t)\leq C(F_0)$. Thus, by technical Lemma \ref{technical-3}, $\sup_{t}\|F_t\|_{\mu_X}\leq B$ for some positive constant $B$. Now,
\begin{align*}
|\mathds E\xi (F_t(X),Y)f_{t+1}(X)|&=|\mathds E\mathds E (\xi(F_t(X),Y)\,|\,X)f_{t+1}(X)|\\
&\leq \mathds E \big|\mathds E(\xi(F_t(X),Y)-\xi(0,Y)\,|\,X)\big| \cdot |f_{t+1}(X)|+\mathds E|\xi(0,Y)f_{t+1}(X)|\\
&\leq L \mathds E|F_{t}(X)f_{t+1}(X)|+\mathds E|\xi(0,Y)f_{t+1}(X)|\\
& \quad \mbox{(by Assumption ${\bf A_3}$)}\\
& \leq L\|F_t\|_{\mu_X}\|f_{t+1}\|_{\mu_X}+\big(\mathds E\xi(0,Y)^2\big)^{1/2}\|f_{t+1}\|_{\mu_X}\\
& \quad \mbox{(by the Cauchy-Schwarz inequality)}\\
& \leq \big( LB +\big(\mathds E\xi(0,Y)^2\big)^{1/2}\big)\|f_{t+1}\|_{\mu_X}.
\end{align*}
Consequently, since $\lim_{t\to \infty}\|f_{t+1}\|_{\mu_X}=0$, 
\begin{align*}
\inf_{f \in \mathscr P} \big(2\mathds E \xi(F_t(X),Y)f(X)+\|f\|_{\mu_X}^2\big)&=2\mathds E \xi(F_t(X),Y)f_{t+1}(X)+\|f_{t+1}\|_{\mu_X}^2\\
&\to 0\mbox{ as }t\to \infty.
\end{align*}
Accordingly, by the symmetry of $\mathscr P$, for all $\varepsilon>0$ and all $t$ large enough, we have, for all $f\in \mathscr P$, 
$$2\mathds E \xi(F_t(X),Y)f(X)+\|f\|_{\mu_X}^2\geq -\varepsilon \quad \mbox{and} -2\mathds E \xi(F_t(X),Y)f(X)+\|f\|_{\mu_X}^2\geq -\varepsilon.$$
So, for all $t$ large enough and all $f\in \mathscr P$,
$$|2\mathds E \xi(F_t(X),Y)f(X)|\leq \varepsilon+ \|f\|_{\mu_X}^2.$$
Since $\varepsilon$ was arbitrary, we conclude that, for all $f\in \mathscr P$,
\begin{equation}
\label{star-algo2}
2 {\limÊ\sup}_{t \to \infty}|\mathds E \xi (F_t(X),Y)f(X)|\leq \|f\|_{\mu_X}^2.
\end{equation}
On the other hand, by Assumption ${\bf A_3}$,
$$|\mathds E(\xi(F_t(X),Y)\,|\,X)|\leq \mathds E(|\xi(0,Y)|\,\big|\,X)+L|F_t(X)|.$$
Since $\sup_t \|F_t\|_{\mu_X}<\infty$, we deduce that
$$\sup_t \|\mathds E(\xi(F_t(X),Y)\,|\,X=\cdot)\|_{\mu_X}<\infty.$$
Therefore, recalling that the unit ball of $L^2(\mu_X)$ is weakly compact, we see that there exists a subsequence
$(F_{t'})_{t'}$ and $\tilde F \in L^2(\mu_X)$ such that, for all $G\in \mbox{lin}(\mathscr P)$,
$$\mathds E \xi (F_{t'}(X),Y)G(X)=\mathds E \mathds E (\xi (F_{t'}(X),Y)\,|\,X)G(X) \to \mathds E \tilde F(X)G(X).$$
Combining this identity with (\ref{star-algo2}) reveals that $2|\mathds E \tilde F(X)f(X)|\leq \|f\|_{\mu_X}^2$ for all $f\in \mathscr P$.
In particular, for all $\varepsilon >0$ and all $f\in \mathscr P$, $2|\mathds E \tilde F(X)\varepsilon f(X)|\leq \varepsilon^2\|f\|_{\mu_X}^2$, and thus, letting $\varepsilon \downarrow 0$, we find that $\mathds E \tilde F(X)f(X)=0$ for all $f\in \mathscr P$. By a linearity argument, we conclude that $\mathds E \tilde F(X)G(X)=0$ for all $G\in \mbox{lin}(\mathscr P)$. Therefore, for all $G\in \mbox{lin}(\mathscr P)$,
\begin{equation}
\label{2star-algo2}
\lim_{t'\to \infty}\mathds E \xi (F_{t'}(X),Y)G(X)= 0,
\end{equation}
which was our first objective.

The next step is to prove that ${\lim\sup}_{t''\to \infty} \mathds E \xi (F_{t''}(X),Y)F_{t''}(X)\leq 0$ for some subsequence $(F_{t''})_{t''}$ of $(F_{t'})_{t'}$. To simplify the notation, we assume, without loss of generality, that $F_0=0$. Fix $\varepsilon >0$. Since $\sum_{k\geq 1}\|f_k\|_{\mu_X}^2<\infty$, there exists $T\geq 0$ such that $\sum_{k \geq T+1} \|f_k\|_{\mu_X}^2\leq \varepsilon$. In addition, for all $t>T$, $F_t=F_T+\nu \sum_{k=T+1}^t f_k$, so that
\begin{equation}
\label{3star-algo2}
\mathds E \xi(F_t(X),Y)F_t(X)=\mathds E\xi(F_t(X),Y)F_T(X)+\nu \sum_{k=T+1}^t \mathds E \xi(F_t(X),Y)f_{k}(X).
\end{equation}
Also, by the very definition of $f_{t+1}$ and the symmetry of $\mathscr P$, we have, for all $f\in \mathscr P$,
\begin{equation}
\label{4star-algo2}
2 \mathds E \xi(F_t(X),Y)f_{t+1}(X)+\|f_{t+1}\|_{\mu_X}^2\leq -2 \mathds E \xi (F_t(X),Y)f(X)+\|f\|_{\mu_X}^2,
\end{equation}
i.e., for all $f\in \mathscr P$,
$$2 \mathds E \xi(F_t(X),Y)f(X)\leq -2 \mathds E \xi (F_t(X),Y)f_{t+1}(X)-\|f_{t+1}\|_{\mu_X}^2+\|f\|_{\mu_X}^2.$$
Using (\ref{3star-algo2}), this leads to
\begin{align}
&\mathds E\xi(F_t(X),Y)F_t(X) \nonumber\\
& \quad \leq \mathds E \xi(F_t(X),Y)F_T(X)+\frac{\nu}{2}\Big(t\big(-2\mathds E\xi(F_t(X),Y)f_{t+1}(X)-\|f_{t+1}\|_{\mu_X}^2\big)+\sum_{k \geq T+1}\|f_k\|_{\mu_X}^2\Big)\nonumber\\
&\quad  \leq \frac{\varepsilon \nu}{2}+\mathds E \xi(F_t(X),Y)F_T(X)+\frac{\nu t}{2}\big(-2\mathds E\xi(F_t(X),Y)f_{t+1}(X)-\|f_{t+1}\|_{\mu_X}^2\big).\label{BLDS}
\end{align}
But, according to inequality (\ref{4star-algo2}) applied with $f=-2f_{t+1}$ (which belongs to $\mathscr P$ by assumption),
$$2\mathds E \xi(F_t(X),Y)f_{t+1}(X)+\|f_{t+1}\|_{\mu_X}^2 \leq 4\mathds E \xi(F_t(X),Y)f_{t+1}(X)+4\|f_{t+1}\|_{\mu_X}^2,$$
i.e.,
$$-2\mathds E \xi(F_t(X),Y)f_{t+1}(X)\leq 3\|f_{t+1}\|_{\mu_X}^2.$$
Combining this inequality with (\ref{BLDS}) shows that
$$\mathds E \xi(F_t(X),Y)F_{t}(X)\leq \frac{\varepsilon \nu}{2}+\mathds E \xi(F_t(X),Y)F_{T}(X)+\nu t\|f_{t+1}\|_{\mu_X}^2.$$
Since $\sum_{k\geq 1} \|f_k\|_{\mu_X}^2<\infty$, one has $\sum_{t'}\|f_{t'+1}\|_{\mu_X}^2<\infty$, which guarantees the existence of a subsequence  $(f_{t''})_{t''}$ satisfying $t''\|f_{t''+1}\|_{\mu_X}^2\to 0$. Besides, since $F_T \in \mbox{lin}(\mathscr P)$, we know from (\ref{2star-algo2}) that $\mathds E \xi (F_{t''}(X),Y)F_{T}(X)\to 0$. Therefore, for all $\varepsilon >0$,
$${\lim\sup}_{t''\to \infty}\mathds E \xi(F_{t''}(X),Y)F_{t''}(X)\leq \frac{\varepsilon \nu}{2}.$$
Since $\varepsilon$ is arbitrary, we have just shown that
\begin{equation}
\label{BP1}
{\lim\sup}_{t''\to \infty} \mathds E \xi (F_{t''}(X),Y)F_{t''}(X)\leq 0,
\end{equation}
as desired.

Let $\varepsilon>0$ and let $F_{\varepsilon}^{\star} \in \mbox{lin}(\mathscr P)$ be such that
$$\inf_{F \in {\mbox{\footnotesize lin}(\mathscr P)}}C(F) \geq C(F^{\star}_{\varepsilon})-\varepsilon.$$
By the convexity of $C$, along $t''$,
\begin{align*}
\inf_{F \in {\mbox{\footnotesize lin}(\mathscr P)}}C(F)&\geq C(F_{\varepsilon}^{\star})-\varepsilon \\
& \geq \inf_k C(F_k)+\mathds E \xi(F_{t''}(X),Y)F_{\varepsilon}^{\star}(X)-\mathds E \xi(F_{t''}(X),Y)F_{t''}(X)-\varepsilon.
\end{align*}
Putting (\ref{2star-algo2}) and (\ref{BP1}) together, we conclude that
$$\lim_{t \to \infty} C(F_t)=\inf_{k} C(F_k) = \inf_{F \in {\mbox{\footnotesize lin}(\mathscr P)}}C(F).$$
\end{proof}

As in \hyperref[algorithm1]{Algorithm 1}, the sequence $(F_t)_t$ approaches $\bar F={\arg\min}_{F \in L^2(\mu_X)}C(F)$, provided $\overline{\mbox{lin}(\mathscr P)}=L^2(\mu_X)$ and 
${\bf A'_3}$ is satisfied in place of ${\bf A_3}$. This is summarized in the following Corollary. Its proof is similar to the proof of Corollary \ref{coro-convergenceCn} and is therefore omitted.
\begin{cor}
\label{coro-convergenceCn-algo2}
Assume that $\overline{\emph{lin}(\mathscr P)}=L^2(\mu_X)$. Assume, in addition, that Assumptions ${\bf A_1}$, ${\bf A_2}$, and ${\bf A'_3}$ are satisfied, and let $(F_t)_t$ be defined by \hyperref[algorithm2]{Algorithm 2} with  $0<\nu<1/(2L)$. Then
$$\lim_{t \to \infty}\|F_t-\bar F\|_{\mu_X}=0,$$
where 
$$\bar F= {\arg\min}_{F \in L^2(\mu_X)} C(F).$$
\end{cor}

Theorem \ref{convergenceCn}/Corollary \ref{coro-convergenceCn} and Theorem \ref{convergenceCn-algo2}/Corollary \ref{coro-convergenceCn-algo2} guarantee that, under appropriate assumptions, \hyperref[algorithm1]{Algorithm 1} and \hyperref[algorithm2]{Algorithm 2} converge towards the infimum of the risk functional. Given the unusual form of these  algorithms, which have the flavor of gradient descents while being different, these results are all but obvious and cannot be deduced from general optimization principles. As far as we know, they are novel in the gradient boosting literature and extend our understanding of the approach. 

Perhaps the most natural framework of \hyperref[algorithm1]{Algorithm 1} and \hyperref[algorithm2]{Algorithm 2} is when $\mu_{X,Y}=\mu_n$, the empirical measure. In this statistical context, both algorithms track the infimum of the empirical risk functional $C_n(F)=\frac{1}{n}\sum_{i=1}^n \psi(F(X_i),Y_i)$
over the linear combinations of weak learners in $\mathscr F$ (\hyperref[algorithm1]{Algorithm 1}) or in $\mathscr P$ (\hyperref[algorithm2]{Algorithm 2}). This task is achieved by sequentially constructing linear combinations of base learners, of the form $F_t=F_0+\sum_{k=1}^t w_k f_k$  with $f_k \in \mathscr F$ for \hyperref[algorithm1]{Algorithm 1}, and $F_t=F_0+\nu \sum_{k=1}^t f_k$ with $f_k \in \mathscr P$ for \hyperref[algorithm2]{Algorithm 2}. We stress that, in the empirical case, the boosted iterates $F_t$ and their eventual limit $\bar F_n$ are measurable functions of the data set $\mathscr D_n$. That being said, Theorem \ref{convergenceCn} and Theorem \ref{convergenceCn-algo2} are numerical-analysis-type results, which do not provide information on the statistical properties of the boosting predictor $\bar F_n$. From this point of view, more or less catastrophic situations can happen, depending on the ``size'' of $\mbox{lin}(\mathscr F)$ (\hyperref[algorithm1]{Algorithm 1}) or $\mbox{lin}(\mathscr P)$ (\hyperref[algorithm2]{Algorithm 2}), which should not be neither too small (to catch complex decisions) nor excessively large (to avoid overfitting).

To be convinced of this, consider for example \hyperref[algorithm1]{Algorithm 1} with $\psi(x,y)=(y-x)^2$ (least squares regression problem) and $\mathscr F=$ all binary trees with $d+1$ leaves. Denote by $P_{n}$ the empirical measure based on the $X_i$ only, $1\leq i\leq n$. Then, by Theorem \ref{convergenceCn}, $\lim_{t \to \infty}C_n(F_t) =C_n(\bar F_n)$, where
$$\bar F_n= {\arg\min}_{F \in L^2(P_{n})} C_n(F).$$
Assume, to simplify, that all $X_i$ are different. It is then easy to see that the boosting predictor $\bar F_n$ takes the value $Y_i$ at each $X_i$ and is arbitrarily defined elsewhere. Of course, in general, such a function $\bar F_n$ does not converge as $n \to \infty$ towards the regression function $F^{\star}(x)=\mathds E(Y|X=x)$, and this is a typical situation where the gradient boosting algorithms overfit. The overfitting issue of boosting procedures has been recognized for a long time, and various approaches have been proposed to combat it, in particular via early stopping \citep[that is, stopping the iterations before convergence; see, e.g.,][]{BuYu03,MaMeZh03,ZhYu05,BiRiZa06,BaTr07}.

Nevertheless, the natural question we would like to answer is whether there exists a reasonable context in which the boosting predictors enjoy good statistical properties as the sample size grows, without resorting to any stopping strategy. The next section provides a positive response. The major constraint we face, imposed by the gradient-descent nature of the algorithms, is that we are required to perform a minimization over a vector space ($\mbox{lin}(\mathscr F)$ for \hyperref[algorithm1]{Algorithm 1} and $\mbox{lin}(\mathscr P)$ for \hyperref[algorithm2]{Algorithm 2}). In particular, there is no question of imposing constraints on the coefficients of the linear combinations, which, for example, cannot reasonably be assumed to be bounded. As we will see, the trick is to carefully constraint the ``complexity'' of the vector spaces $\mbox{lin}(\mathscr F)$ or $\mbox{lin}(\mathscr P)$ in a manner compatible with the algorithms. The second message is the importance of having a strongly convex functional risk to minimize, which, in some way, restrict the norm of the sequence $(F_t)_{t \geq 0}$ of boosted iterates. As we have pointed out several times, if the loss function is not natively strongly convex in its first argument, then this type of regularization can be achieved by resorting to an $L^2$-type penalty.
\section{Large sample properties}
\label{section4}
We consider in this section a functional minimization problem whose solution can be computed by gradient boosting and enjoys non-trivial statistical properties. The context and notation are similar to that of the previous sections, but must be slightly adapted to fit this new framework.

For simplicity, it will be assumed throughout that $\mathscr X$ is a compact subset of $\mathds R^d$. We consider i.i.d.~data $\mathscr D_n=\{(X_1,Y_1), \hdots, (X_n,Y_n)\}$ taking values in $\mathscr X \times \mathscr Y$, and let $P_n$ be the empirical measure based on the $X_i$ only, $1\leq i\leq n$. We denote by $P$ the common distribution of the $X_i$ and assume that $P$ has a density $g$ with respect to the Lebesgue measure $\lambda$ on $\mathds R^d$, with
$$0<\inf_{\mathscr X} g\leq \sup_{\mathscr X} g<\infty.$$
We concentrate on \hyperref[algorithm1]{Algorithm 1} and take as weak learners a finite class $\mathscr F_n$ of simple functions on $\mathscr X$ with $\pm 1$ values, which may possibly vary with the sample size $n$. It is actually easy to verify that all subsequent results are valid for \hyperref[algorithm2]{Algorithm 2} by letting $\mathscr P_n=\{\lambda f: f\in \mathscr F_n, \lambda \in \mathds R\}$.

The typical example we have in mind for $\mathscr F_n$ is a finite class of binary trees using axis parallel cuts with $k$ leaves. Of course, the parameter $k$ has to be carefully chosen as a function of the sample size to guarantee consistency, as we will see below. The fact that the class $\mathscr F_n$ is supposed to be finite should not be too disturbing, since in practice the optimization step (\ref{optimizationstep}) is typically performed over a finite family of functions. This is for example
the case when a CART-style top-down recursive partitioning is used to compute the minimum at each iteration of the algorithm. In this approach, the optimal tree in (\ref{optimizationstep}) is greedily searched for by passing from one level of node to the next one with cuts that are located between two data points. So, even though the collection $\mathscr F_n$ may be very large, it is nevertheless fair to assume that its cardinal is finite.

As before, it is assumed that the identically zero function belongs to $\mathscr F_n$. So, in this framework, we see that there exists a (large) integer $N=N(n)\geq 1$ and a partition of $\mathscr X$ into measurable subsets $A_j^n$, $1\leq j \leq N$, such that any $F\in \mbox{lin}(\mathscr F_n)$ takes the form
$F=\sum_{j=1}^N\alpha_j\mathds 1_{A_j^n}$, where $(\alpha_1, \hdots, \alpha_N)\in \mathds R^N$. To avoid pathological situations, we assume that there exists a positive sequence $(v_n)_{n}$ such that $\min_{1\leq j \leq N} \lambda(A_j^n) \geq v_n$. Of course, it is supposed that $N\to \infty$ as $n$ tends to infinity. 

We let $\phi:\mathds R\times \mathscr Y \to \mathds R_+$ be a loss function, assumed to be convex in its first argument and to satisfy $\bar \phi:=\sup_{y \in \mathscr Y}\phi(0,y)<\infty$. In line with the previous sections, we are interested in minimizing over $\mathscr F_n$ the empirical risk functional $C_n(F)$ defined by
$$C_{n}(F)=\frac{1}{n}\sum_{i=1}^n \psi(F(X_i),Y_i),$$
where $\psi(x,y)=\phi(x,y)+\gamma_n x^2$ and $(\gamma_n)_n$ is a sequence of positive parameters such that $\lim_{n \to \infty}\gamma_n=0$. (Note that $\gamma_n$ depends only on $n$ and is therefore kept fixed during the iterations of the algorithm.)
Put differently,
\begin{equation}
\label{An+P}
C_{n}(F)=A_{n}(F)+\gamma_n\|F\|^2_{P_n},
\end{equation}
where 
$$A_{n}(F)=\frac{1}{n}\sum_{i=1}^n \phi(F(X_i),Y_i).$$ 
Assumption ${\bf A_1}$ is obviously satisfied (with $\mu_{X,Y}=\mu_n$, in the notation of Section \ref{section3}), and the same is true for Assumption ${\bf A_2}$ by the $\alpha$-strong convexity of the function $\psi(\cdot, y)$ for each fixed $y$, with $\alpha$ independent of $y$.
\begin{rem}
\label{rem1}
If the function $\phi(\cdot, y)$ is natively $\alpha$-strongly convex with a parameter $\alpha$ independent of $y$, then we may consider the simplest problem of minimizing the functional $A_n(F)$. Indeed, in this case there is no need to resort to the $\gamma_n\|F\|^2_{P_n}$ penalty term since Lemma \ref{technical-3} allows to bound $\|F\|^2_{P_n}$. As we have seen in Section \ref{section2}, this is for example the case in the least squares problem, when $\phi(x,y)=(y-x)^2$. However, to keep a sufficient degree of generality, we will consider in the following the more general optimization problem (\ref{An+P}). 
\end{rem}
Now, let
$$\bar F_n ={\arg\min}_{F \in{\mbox{\footnotesize lin}(\mathscr F_n)}}C_n(F).$$
We have learned in Theorem \ref{convergenceCn} that whenever Assumption ${\bf A_3}$ is satisfied, the boosted iterates $(F_t)_t$ of \hyperref[algorithm1]{Algorithm 1} satisfy $\lim_{t\to \infty} C_n(F_t)=C_n(\bar F_n)$, i.e.,
$$\lim_{t\to \infty} \big(A_{n}(F_t)+\gamma_n \|F_t\|^2_{P_n}\big)=A_{n}(\bar F_{n})+\gamma_n\|\bar F_{n}\|_{P_n}^2.$$
For $F\in L^2(P)$, the population counterpart of $A_n(F)$ is the convex functional $A(F):=\mathds E\phi(F(X_1),Y_1)$, which is assumed to be locally bounded, and thus continuous. Throughout, we denote by $F^{\star}$ a minimizer of $A(F)$ over $L^2(P)$, i.e.,
$$
F^{\star} \in {\arg\min}_{F \in L^2(P)} A(F).
$$
We have for example $F^{\star}(x)=\mathds E(Y|X=x)$ in the regression problem with $\phi(x,y)=(y-x)^2$ and $F^{\star}(x)=\log (\frac{\eta(x)}{1-\eta(x)})$ in the classification problem with $\phi(x,y)=\log_2(1+e^{-yx})$, where $\eta(x)=\mathds P(Y=1|X=x)$. 

Our goal in this section is to investigate the large sample properties of $\bar F_{n}$, i.e., to analyze the statistical behavior of the boosting predictor $\bar F_{n}$ as $n\to \infty$. In particular, a sensible objective is to show that $A(\bar F_{n})$ gets asymptotically close to the minimal risk $A(F^{\star})$ as the sample size grows. This necessitates a proof, since all we know for now is that 
$$A_{n}(\bar F_{n})+\gamma_n \|\bar F_n\|^2_{P_n}-A({F^\star}) = \inf_{F \in \mbox{\footnotesize lin}(\mathscr F_n)}\big(A_{n}(F)+\gamma_n \|F\|^2_{P_n}-A({F^\star})\big),$$
which is our starting point. The following assumption on $\phi$ will be needed in the analysis:
\begin{enumerate}
\item[${\bf A_4}$] For all $p\geq 0$, there exists a constant $\zeta(p)> 0$ such that, for all $(x_1,x_2,y) \in \mathds R^2\times \mathscr Y$ with $\max(|x_1|,|x_2|)\leq p$,
$$|\phi(x_1,y)-\phi(x_2,y)| \leq \zeta(p) |x_1-x_2|.$$
\end{enumerate}
It is readily seen that all classical convex losses in regression and classification satisfy this local Lipschitz assumption. Finally, we let $A^n(x)=A_j^n$ whenever $x\in A_j^n$, and, for $E\subset \mathds R^d$,
$$\mbox{diam}(E)=\sup_{x,x' \in E}\|x-x'\|.$$
Recall that $\bar \phi:=\sup_{y \in \mathscr Y}\phi(0,y)<\infty$. 
\begin{theo}
\label{theorem-stat-consistency}
Assume that Assumptions ${\bf A_3}$ (with $\psi(x,y)=\phi(x,y)+\gamma_nx^2$) and ${\bf A_4}$ are satisfied, and that $F^{\star}$ is bounded. Assume, in addition, that $\emph{diam}(A^n(X))\to 0$ in probability as $n\to \infty$. Then, provided $\gamma_n\to0$, $N\to \infty$, $\frac{\log N}{nv_n}\to 0$, and
$$
\frac{1}{\sqrt{nv_n\gamma_n}}\zeta \bigg(\sqrt{\frac{2\bar \phi}{v_n \gamma_n\inf_{\mathscr X} g}}\bigg) \to 0,
$$
we have
$\lim_{n \to \infty}\mathds E A(\bar F_n) =A(F^{\star})$.
\end{theo}

The main message of this theorem is that, under appropriate conditions on the loss and provided the size of the weak learner classes are judiciously increased, gradient boosting does not overfit. In other words, in this framework, stopping the iterations is not necessary and the algorithms may be run indefinitely, without worrying about early stopping issues.

In line with Remark \ref{rem1}, we leave it as an exercise to prove that if the function $\phi(\cdot, y)$ is already $\alpha$-strongly convex with a parameter $\alpha$ independent of $y$, then a similar result holds with the conditions $N\to \infty$, $\frac{\log N}{nv_n}\to 0$, and
$$
\frac{1}{\sqrt{nv_n}}\zeta \bigg(\sqrt{\frac{a}{v_n \inf_{\mathscr X} g}}\bigg) \to 0,
$$
where $a=\frac{2}{\alpha}\sup_{y \in \mathscr Y} |\xi(0,y)|+\sqrt{2\bar \phi / \alpha}$. 
In this case, we can take $\gamma_n=0$ (i.e., no penalty) and resort to Lemma \ref{technical-3} to bound the quantity $\|F\|^2_{P_n}$. 

Next, we point out that the conditions of Theorem \ref{theorem-stat-consistency} are mild and cover a wide variety of losses and possible classes of weak learners. As an example, let $\mathscr X=[0,1]^d$ and take for $\mathscr F_n$ the set of all binary trees on $[0,1]^d$ with $k_n$ leaves, where cuts are perpendicular to the axes and located at the middle of the cells. Although combinatorially rich, this family of trees is finite (see Figure \ref{figure1} for an illustration in dimension $d=2$). 
\begin{figure}[!h]
\begin{center}
\includegraphics[scale=0.34]{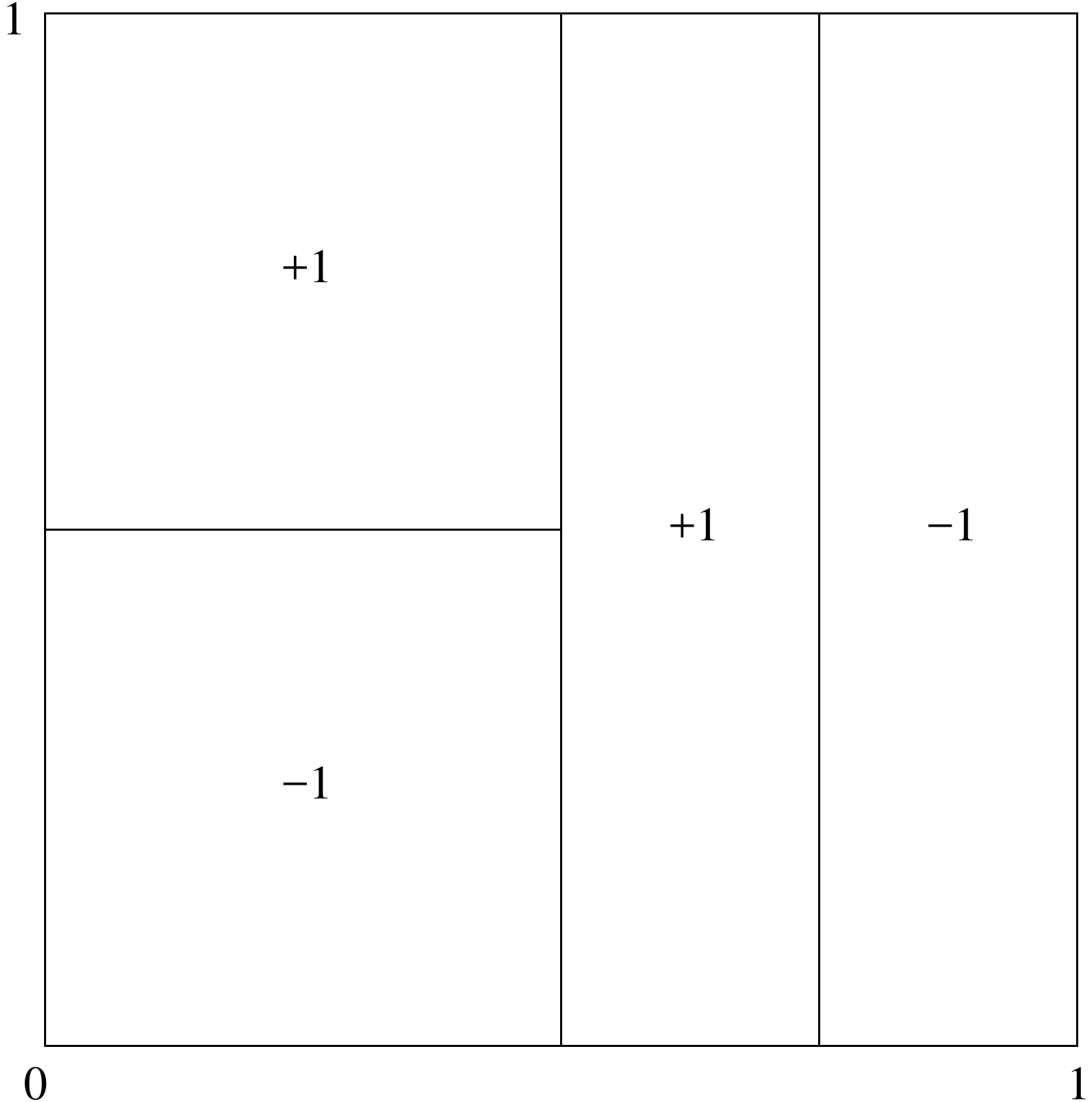}
\includegraphics[scale=0.34]{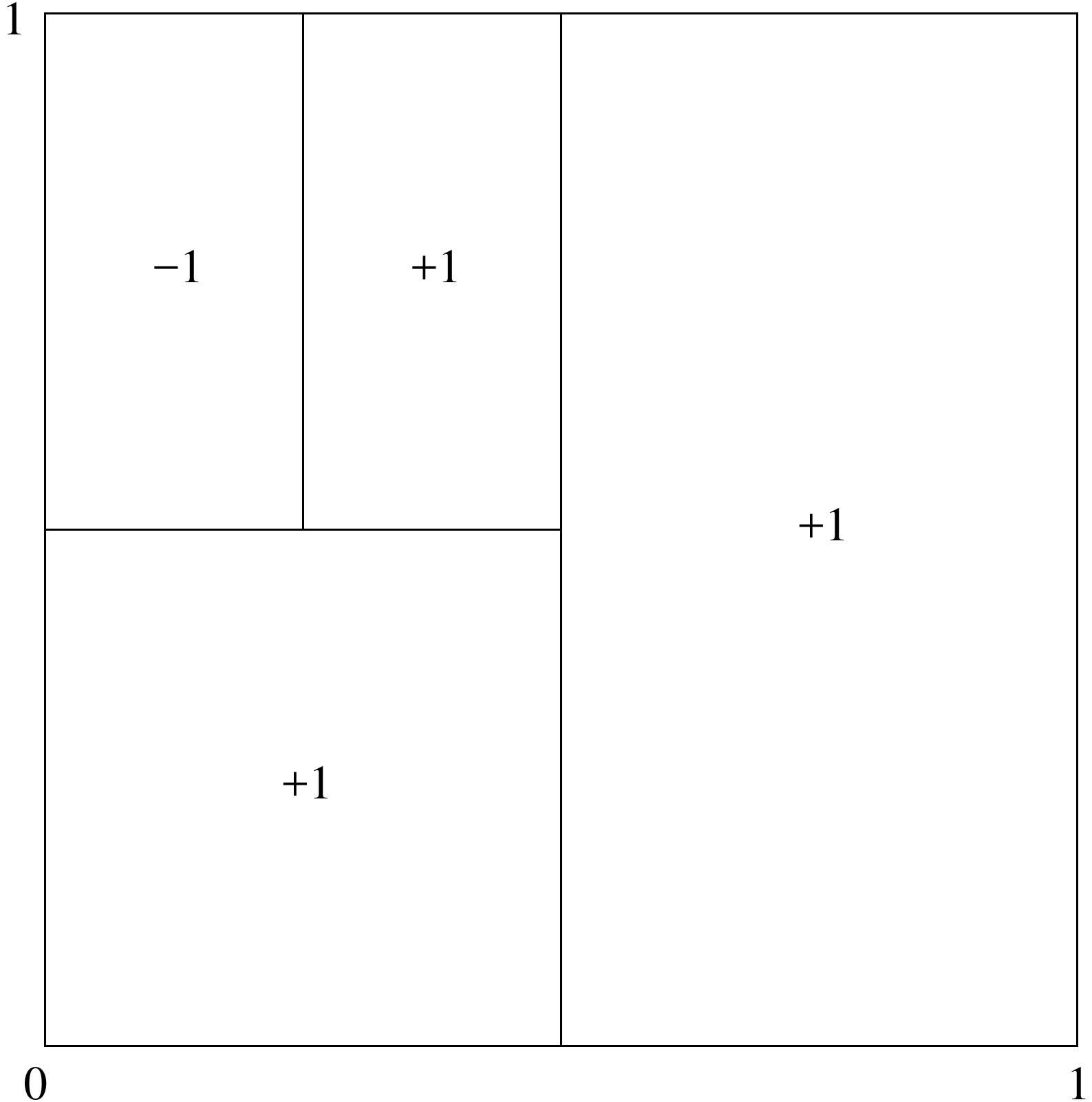}
\includegraphics[scale=0.34]{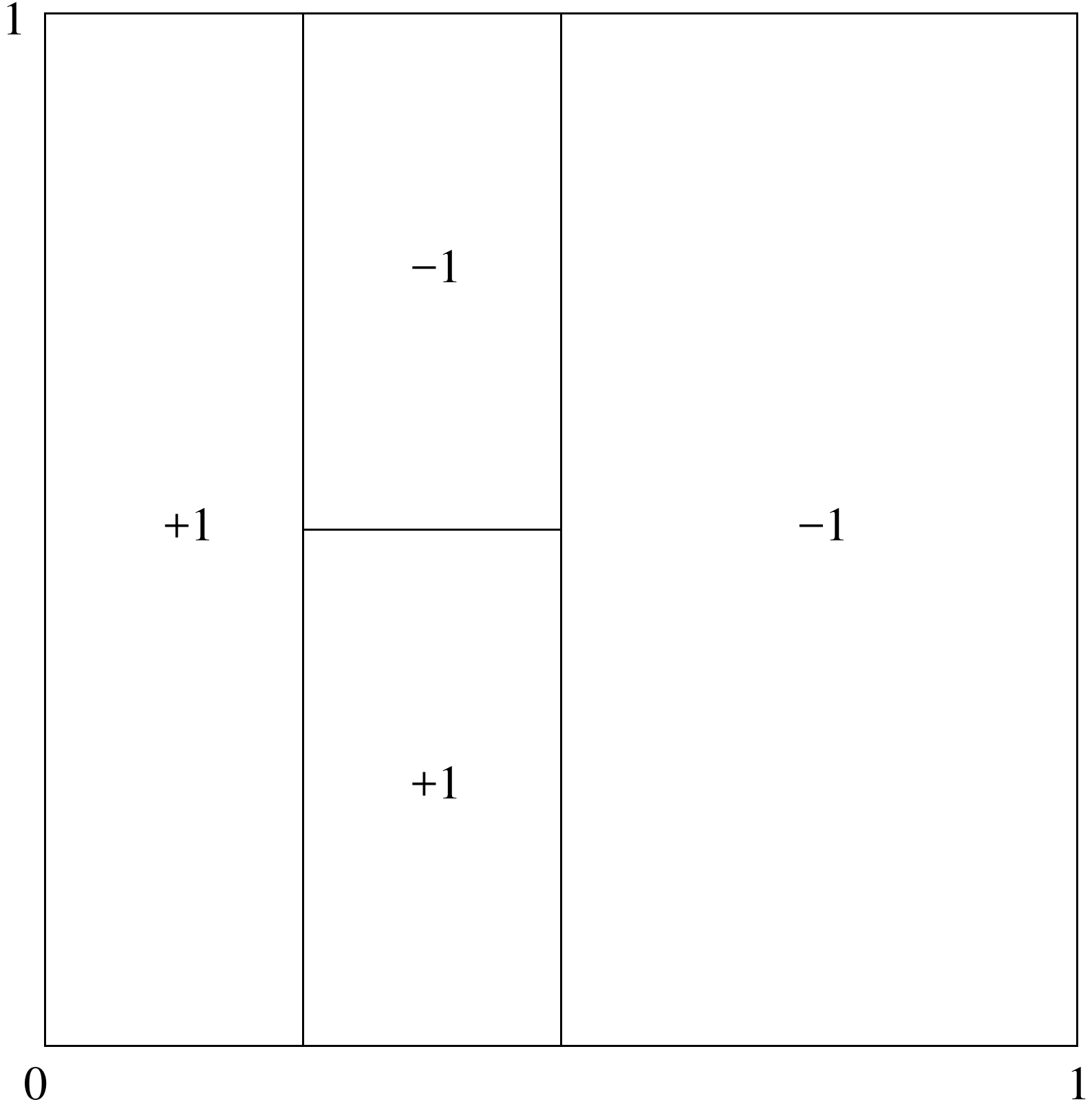}
\includegraphics[scale=0.34]{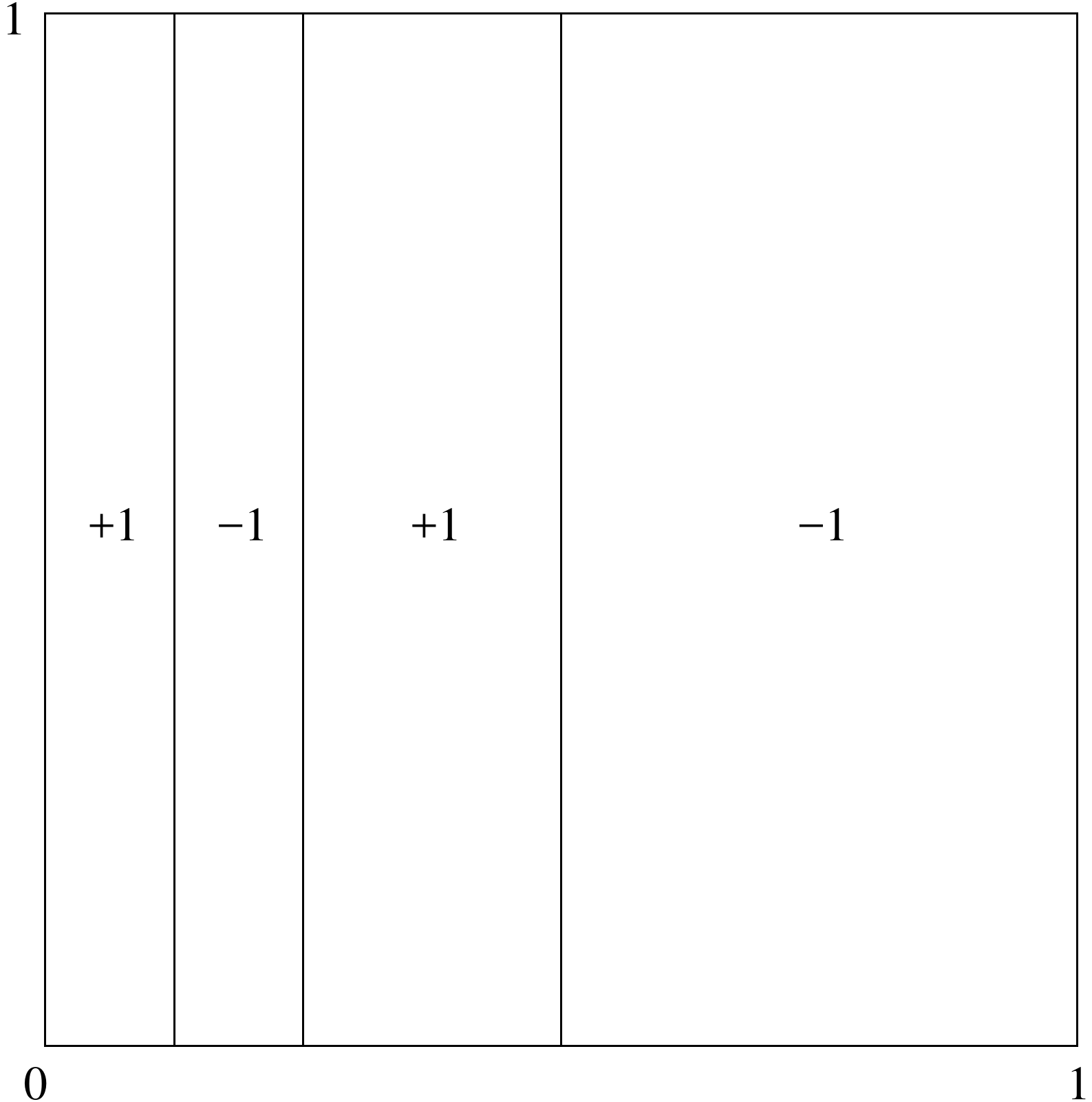}
\end{center}
\caption{Four examples of trees in the class $\mathscr F_n$, in dimension $d=2$, with $k_n=4$.}\label{figure1}
\end{figure}

It is easy to verify that any $F\in \mbox{lin}(\mathscr F_n)$ takes the form $F=\sum_{j=1}^N\alpha_j\mathds 1_{A_j^n}$, where $N= 2^{dk_n}$ and the $A_j^n$, $1\leq j \leq N$, form a regular grid over $[0,1]^d$. Thus, clearly, $v_n= 2^{-dk_n}$. In addition, considering for example the loss $\phi(x,y)=(y-x)^2$, we see that the conditions of Theorem \ref{theorem-stat-consistency} take the simple form 
$$k_n\to \infty, \quad \frac{k_n2^{dk_n}}{n}\to 0, \quad \mbox{and}\quad \frac{2^{dk_n}}{\sqrt n}\to 0.$$
Let us finally note that in the $\pm 1$-classification setting, each $F$ defines a classifier $g_F$ in a natural way, by
\begin{equation*} 
g_F(x) = \left\{
\begin{array}{ll}
1 & \mbox{if $F(x)>0$}\\
-1 & \mbox{otherwise,}
\end{array}
\right.
\end{equation*} 
and the main concern is not the behavior of the theoretical risk $A(F)$ with respect to $A(F^{\star})$, but rather the proximity between the probability of error $L(g_F):=\mathds P(g_F(X)\neq Y)$ and the Bayes risk $L^{\star}:=\inf_{g:\mathscr X\to \{-1,1\}} \mathds P(g(X)\neq Y)$. For most classification losses \citep{Zh04,BaJoAu06}, the difference $L(g_F)-L^{\star}$ is small as long as $A(F) -A(F^{\star})$ is. In our framework, we conclude that for such well-behaved losses, under the assumptions of Theorem \ref{theorem-stat-consistency}, 
$$\lim_{n \to \infty}\mathds EL(g_{\bar F_{n}}) = L^{\star}.$$
\begin{proof}[Proof of Theorem \ref{theorem-stat-consistency}]
For $\beta \in \mathds R^N$, we let $F_{\beta}=\sum_{j=1}^N \beta_j\mathds 1_{A_j^n}$ and notice that $\bar F_n=F_{\alpha}$ for some (data-dependent) $\alpha \in \mathds R^N$. Let the event $S$ be defined by
$$S=\big\{\forall j=1, \hdots, N: P_n(A_j^n)\geq P(A_j^n)/2\big\}.$$
Observe that
$$\|\bar F_n\|^2_{P_n} \leq \frac{\frac{1}{n}\sum_{i=1}^n\phi(0,Y_i)}{\gamma_n}\leq \frac{\bar \phi}{\gamma_n},$$
and, similarly, that
$$\|\bar F_n\|^2_{P_n}= \sum_{j=1}^N\alpha_j^2P_n(A_j^n).$$
Therefore, on $S$, 
$$\frac{1}{2}\sum_{j=1}^N \alpha_j^2P(A_j^n)\leq \frac{\bar \phi}{\gamma_n},$$ 
and so
$$\frac{\inf_{\mathscr X} g}{2} \cdot v_n \sum_{j=1}^N \alpha_j^2 \leq \frac{\bar \phi}{\gamma_n}.$$
We have just shown that, on the event $S$, $\alpha \in T$, 
where
$$T=\Big\{ \beta \in \mathds R^N:\sum_{j=1}^N\beta_j^2\leq \frac{2\bar \phi}{\inf_{\mathscr X}g}\cdot \frac{1}{v_n\gamma_n}\Big\}.$$
Now, observe that
\begin{align*}
\mathds EC_n(\bar F_n)&= \mathds E\inf_{F \in \mbox{\footnotesize lin}(\mathscr F_n)}C_n(F) \\
&= \mathds E\inf_{F \in \mbox{\footnotesize lin}(\mathscr F_n)}C_n(F)\mathds 1_S+ \mathds E\inf_{F \in \mbox{\footnotesize lin}(\mathscr F_n)}C_n(F)\mathds 1_{S^c}\\
&\leq \mathds E\inf_{F \in \mbox{\footnotesize lin}(\mathscr F_n)}C_n(F)\mathds 1_S+\mathds EC_n(0)\mathds 1_{S^c}\\
& = \mathds E \inf_{\beta \in T} C_n(F_{\beta})\mathds 1_S+\mathds EA_n(0)\mathds 1_{S^c}\\
& \leq \mathds E \inf_{\beta \in T} C_n(F_{\beta})+\bar \phi \mathds P(S^c).
\end{align*}
Define 
$$D_n(F)=A(F)+\gamma_n\|F\|^2_{P_n}.$$
Since $C_n(F)-D_n(F)=A_n(F)-A(F)$, we deduce from Lemma \ref{VCfinite} and Lemma \ref{estimationlemma} that whenever 
$$\frac{\log N}{nv_n}\to 0 \quad \mbox{and} \quad \frac{1}{\sqrt{nv_n\gamma_n}}\zeta \bigg(\sqrt{\frac{2\bar \phi}{v_n \gamma_n\inf_{\mathscr X} g}}\bigg) \to 0,$$
we have 
\begin{align}
\limsup_{n \to \infty} \mathds EC_n(\bar F_n) &\leq \limsup_{n \to \infty} \mathds E \inf_{\beta \in T} C_n(F_{\beta})\nonumber\\
& \leq \limsup_{n \to \infty} \mathds E \inf_{\beta \in T} D_n(F_{\beta}) +\limsup_{n \to \infty}  \big(\mathds E \sup_{\beta \in T} |A_n(F_{\beta})-A(F_{\beta})|\big)\nonumber\\
&=\limsup_{n \to \infty} \mathds E \inf_{\beta \in T} D_n(F_{\beta}).\label{BLV}
\end{align}
Let $\varepsilon>0$. By Lemma \ref{approxlemma}, there exists $(\beta_1^{\varepsilon}, \hdots, \beta_N^{\varepsilon}) \in T$ such that
$$\big\|F^{\star}-\sum_{j=1}^N\beta_j^{\varepsilon} \mathds 1_{A_j^n}\big\|_{P}\leq \varepsilon.$$
Define $F_{\varepsilon}^{\star}=\sum_{j=1}^N\beta_j^{\varepsilon}\mathds 1_{A_j^n}$. Then, according to (\ref{BLV}),
\begin{align}
\limsup_{n \to \infty} \mathds E C_n(\bar F_n) & \leq \limsup_{n \to \infty}\big( A(F_{\varepsilon}^{\star})+\gamma_n \mathds E \|F_{\varepsilon}^{\star}\|^2_{P_n}\big)\nonumber\\
&=  \limsup_{n \to \infty}\big( A(F_{\varepsilon}^{\star})+\gamma_n \mathds \|F_{\varepsilon}^{\star}\|^2_{P}\big)\nonumber\\
& \leq A(F^{\star}_{\varepsilon}).\label{DD14J}
\end{align}
Since $A$ is continuous, we conclude that $\limsup_{n \to \infty} \mathds E C_n (\bar F_n)\leq A(F^{\star})$.

On the other hand, $C_n(\bar F_n)\geq A_n(\bar F_n)$, and, by Lemma \ref{VCfinite} and Lemma \ref{estimationlemma},
\begin{align*}
\mathds E |A_n(\bar F_n)-A(\bar F_n)| &\leq \mathds E \sup_{\beta \in T} |A_n(F_{\beta})-A(F_{\beta})|+\bar \phi  \mathds P(S^c) \\
&\to 0 \quad \mbox{as }n \to \infty.
\end{align*}
Therefore,
$$\limsup_{n \to \infty} \mathds E A(\bar F_n) \leq \limsup_{n \to \infty}C_n(\bar F_n).$$
So, with (\ref{DD14J}),
$$\limsup_{n \to \infty} \mathds E A(\bar F_n) \leq A(F^{\star}),$$
which is the desired result.
\end{proof}
\section{Some technical lemmas}
\begin{lem}
\label{technical-1}
Assume that Assumptions ${\bf A_1}$ and ${\bf A_3}$ are satisfied. Then, for all $a>0$ and all $F,G \in L^2(\mu_X)$, 
$$C(F)-C(F+aG) \geq -a^2L\|G\|_{\mu_X}^2-a\mathds E \xi (F(X),Y)G(X).$$
\begin{proof}
By inequality (\ref{sub-zero}),
\begin{align*}
C(F)&\geq C(F+aG)-a\mathds E \xi (F(X)+aG(X),Y)G(X)\\
&=C(F+aG)-a\mathds E (\xi(F(X)+aG(X),Y)-\xi(F(X),Y))G(X)\\
& \quad-a\mathds E \xi(F(X),Y)G(X)\\
& = C(F+aG)-a\mathds E\mathds E (\xi(F(X)+aG(X),Y)-\xi(F(X),Y)\,|\,X)G(X)\\
& \quad-a\mathds E \xi(F(X),Y)G(X)\\
&\geq C(F+aG)-a\big(\mathds E\mathds E^2 (\xi(F(X)+aG(X),Y)-\xi(F(X),Y)\,|\,X)\big)^{1/2}\|G\|_{\mu_X}\\
& \quad-a\mathds E \xi(F(X),Y)G(X)\\
& \quad \mbox{(by the Cauchy-Schwarz inequality).}
\end{align*}
Thus, by Assumption ${\bf A_3}$, 
$$C(F)\geq C(F+aG) -a^2L\|G\|_{\mu_X}^2-a\mathds E \xi (F(X),Y)G(X).$$
\end{proof}
\begin{lem}
\label{technical-2}
Assume that Assumption ${\bf A_1}$ is satisfied, and let $(F_t)_t$ be defined by \hyperref[algorithm1]{Algorithm 1} with $(w_t)_t$ as in (\ref{choixW}). If, for some $t_0\geq 0$,
$$\mathds E \xi (F_{t_0}(X),Y)f(X)=0 \quad \mbox{for all } f\in \mathscr F,$$
then $C(F_{t_0})=\inf_{F \in \emph{\footnotesize lin}(\mathscr F)}C(F)$.
\end{lem}
\begin{proof}
Fix $t_0\geq 0$ and assume that $\mathds E \xi (F_{t_0}(X),Y)f(X)=0$ for all $f\in \mathscr F$. By linearity, $\mathds E \xi (F_{t_0}(X),Y)G(X)=0$ for all $G\in \mbox{lin}(\mathscr F)$. Let $\varepsilon>0$ and let $F_{\varepsilon}^{\star} \in \mbox{lin}(\mathscr F)$ be such that
$$\inf_{F \in \mbox{\footnotesize lin}(\mathscr F)}C(F)\geq C(F_{\varepsilon}^{\star})-\varepsilon.$$
By the convexity inequality (\ref{sub-zero}),
$$
C(F_{\varepsilon}^{\star})\geq C(F_0)+\mathds E \xi (F_{t_0}(X),Y)(F_{\varepsilon}^{\star}(X)-F_{t_0}(X))=C(F_{t_0}).
$$
Thus,
$$\inf_{F \in \mbox{\footnotesize lin}(\mathscr F)}C(F)\geq C(F_{t_0})-\varepsilon.$$
Since $\varepsilon$ is arbitrary, the result follows.
\end{proof}
\end{lem}
\begin{lem}
\label{technical-3}
Assume that Assumptions ${\bf A_1}$ and ${\bf A_2}$ are satisfied. Then, for all $F \in L^2(\mu_X)$,
$$\|F\|_{\mu_X} \leq \frac{2}{\alpha} \big(\mathds E \xi (0,Y)^2\big)^{1/2}+\sqrt{\frac{2C(F)}{\alpha}}.$$
\end{lem}
\begin{proof}
By inequality (\ref{sub-one}) and the Cauchy-Schwarz inequality,
\begin{align*}
C(F) & \geq C(0)+\mathds E \xi(0,Y)F(X)+\frac{\alpha}{2}\|F\|_{\mu_X}^2\\
& \geq C(0)-\big(\mathds E \xi (0,Y)^2\big)^{1/2}\|F\|_{\mu_X}+\frac{\alpha}{2}\|F\|_{\mu_X}^2.
\end{align*}
Let $\kappa=(\mathds E \xi(0,Y)^2)^{1/2}$. Since $C(0)\geq 0$, 
$$C(F)+\kappa\|F\|_{\mu_X}-\frac{\alpha}{2}\|F\|_{\mu_X}^2\geq 0.$$
Therefore,
$$
\|F\|_{\mu_X} \leq \frac{\kappa+\sqrt{\kappa^2+2\alpha C(F)}}{\alpha}\leq \frac{2\kappa}{\alpha}+\sqrt{\frac{2C(F)}{\alpha}}.
$$
\end{proof}
\begin{lem}
\label{VCfinite}
Let the event $S$ be defined by
$$S=\big\{\forall j=1, \hdots, N: P_n(A_j^n)\geq P(A_j^n)/2\big\}.$$
If $\frac{\log N}{nv_n}\to 0$, then $\lim_{n \to \infty} \mathds P(S^c)=0$.
\end{lem}
\begin{proof}
We have
\begin{align*}
\mathds P(S^c)&=\mathds P\big(\exists j \leq N: P_n(A_j^n)< P(A_j^n)/2\big)\\
&=\mathds P\big(\exists j \leq N: P_n(A_j^n)-P(A_j^n)< -P(A_j^n)/2\big)\\
&=\mathds P\Big(\exists j \leq N: \frac{P(A_j^n)-P_n(A_j^n)}{\sqrt{P(A_j^n)}}> \sqrt{P(A_j^n)}/2\Big)\\
& \leq \mathds P\Big(\max_{1\leq j \leq N} \frac{P(A_j^n)-P_n(A_j^n)}{\sqrt{P(A_j^n)}}> \sqrt{v_n\inf_{\mathscr X}g}/2\Big)\\
& \leq c_1N e^{-nv_n\inf_{\mathscr X} g/c_2},
\end{align*}
where $c_1$ and $c_2$ are positive constants. In the last inequality, we used a Vapnik-Cher\-vo\-nen\-kis inequality \citep{Va88} for relative deviations.
\end{proof}
In the sequel, we let 
$$T=\Big\{ \beta \in \mathds R^N:\sum_{j=1}^N\beta_j^2\leq \frac{2\bar \phi}{\inf_{\mathscr X}g}\cdot \frac{1}{v_n\gamma_n}\Big\},$$
where $\bar \phi=\sup_{\mathscr Y}\phi(0,y)<\infty$. We recall that $A^n(x):=A_j^n$ whenever $x\in A_j^n$.
\begin{lem}
\label{approxlemma}
Assume that $\emph{diam}(A^n(X))\to 0$ in probability and that $\gamma_n \to 0$ as $n\to \infty$. For all $\varepsilon >0$ and all $n$ large enough, there exists $(\beta_1^{\varepsilon}, \hdots, \beta_N^{\varepsilon})\in T$ such that
$$\big\| F^{\star}-\sum_{j=1}^N \beta_j^{\varepsilon}\mathds 1_{A_j^n}\big\|_{P}\leq \varepsilon.$$
\end{lem}
\begin{proof}
Let $K$ be a  bounded and uniformly continuous function on $\mathds R^d$, with $\int K \rm{d}\lambda =1$. Let, for $p>0$, 
$$K_p(x)=p^dK\Big(\frac{x}{p}\Big), \quad x \in \mathds R^d.$$
With a slight abuse of notation, we consider $F^{\star}$ as a function defined on the whole space $\mathds R^d$ (instead of $\mathscr X$) by  implicitly assuming that $F^{\star}=0$ on $\mathscr X^c$. We also define $ F_p^{\star}=F^{\star} \star K_p$, i.e.,
$$F^{\star}_p(x)=\int_{\mathds R^d} K_p(z)F^{\star}(x-z){\rm d}z, \quad x\in \mathds R^d.$$
Let $(L^2(\lambda), \|\cdot\|_{\lambda})$ be the vector space of all real-valued square integrable functions on $\mathds R^d$. For all $p$ large enough, we have
$$\|F_p^{\star}-F^{\star}\|_{\lambda}\leq \frac{\varepsilon}{2\sqrt{\sup_{\mathscr X}g}}$$
\citep[see, e.g.,][Theorem 9.6]{WhZy77}. Therefore, for all $p$ large enough,
\begin{equation}
\label{MCM1}
\|F_p^{\star}-F^{\star}\|_{P}\leq \varepsilon/2.
\end{equation}
In addition, $F^{\star}_p$ is uniformly continuous on $\mathscr X$ \citep[][Theorem 9.4]{WhZy77}. Thus, there exists $\eta=\eta(\varepsilon,p)>0$ such that, for all $(x,x')\in \mathscr X^2$ with $\|x-x'\|\leq \eta$,
$$|F^{\star}_p(x)-F^{\star}_p(x')| \leq \varepsilon/\sqrt{8}.$$
For each $j\in \{1, \hdots, N\}$, choose an arbitrary $a_j^n \in A_j^n$ and set $G_p^{\star}=\sum_{j=1}^NF^{\star}_p(a_j^n)\mathds 1_{A_j^n}$.
Then
\begin{align*}
\|G^{\star}_p-F^{\star}_p\|^2_{P}&=\sum_{j=1}^N\mathds E(G_p^{\star}(X)-F^{\star}_p(X))^2\mathds 1_{[X\in A_j^n, \mbox{\footnotesize diam}(A^n(X))\leq \eta]}\\
&\quad +\sum_{j=1}^N\mathds E(G_p^{\star}(X)-F^{\star}_p(X))^2\mathds 1_{[X\in A_j^n, \mbox{\footnotesize diam}(A^n(X))> \eta]}\\
&=\sum_{j=1}^N\mathds E(F^{\star}_p(a_j^n)-F_p^{\star}(X))^2\mathds 1_{[X\in A_j^n, \mbox{\footnotesize diam}(A^n(X))\leq \eta]}\\
&\quad +\sum_{j=1}^N\mathds E(G_p^{\star}(X)-F^{\star}_p(X))^2\mathds 1_{[X\in A_j^n, \mbox{\footnotesize diam}(A^n(X))> \eta]}\\
& \leq \frac{\varepsilon^2}{8}\sum_{j=1}^N \mathds P(X\in A_j^n)+4\sup_{\mathscr X}(F^{\star})^2\sum_{j=1}^N\mathds P(X\in A_j^n, \mbox{diam}(A^n(X))> \eta)\\
& \quad \mbox{(since $\sup_{\mathscr X}|F^{\star}_p| \leq \sup_{\mathscr X}|F^{\star}|<\infty$ and $\sup_{\mathscr X}|G^{\star}_p| \leq \sup_{\mathscr X}|G^{\star}|<\infty$)}\\ 
&\leq \frac{\varepsilon^2}{8}+4\sup_{\mathscr X}(F^{\star})^2\mathds P(\mbox{diam}(A^n(X))> \eta),
\end{align*}
because the $(A_j^n)_{1\leq j \leq N}$ form a partition of $\mathscr X$. Since $\mbox{diam}(A^n(X))\to 0$ in probability, we see that for all $n$ large enough (depending upon $\varepsilon$ and $p$),
$$\|G^{\star}_p-F^{\star}_p\|_{P}\leq \varepsilon/2.$$
Letting $\beta_j^{\varepsilon}=F^{\star}_p(a_j^n)$, $1\leq j\leq N$, and combining this inequality and (\ref{MCM1}), we conclude that for every fixed $\varepsilon >0$ and all $n$ large enough, there exists $(\beta_1^{\varepsilon}, \hdots, \beta_N^{\varepsilon}) \in \mathds R^N$ such that
$$\big\|F^{\star}-\sum_{j=1}^N\beta_j^{\varepsilon}\mathds 1_{A_j^n}\big\|_{P}\leq \varepsilon.$$
To complete the proof, it remains to show that $(\beta_1^{\varepsilon}, \hdots, \beta_N^{\varepsilon}) \in T$. Observe that
$$\sum_{j=1}^N (\beta_j^{\varepsilon})^2\leq \sup_{\mathscr X}(F^{\star})^2N.$$
The right-hand side is bounded by $\frac{2\bar \phi}{\inf_{\mathscr X}g }\cdot \frac{1}{v_n\gamma_n}$ for all $n$ large enough. To see this, just note that 
$$Nv_n \leq \sum_{j=1}^N\lambda(A_j^n)=\lambda(\mathscr X)<\infty.$$
Therefore,
$Nv_n\gamma_n \leq \lambda(\mathscr X)\gamma_n \to 0$ as $n \to \infty$. This concludes the proof of the lemma. 
\end{proof}
\begin{lem}
\label{estimationlemma}
For $\beta \in \mathds R^N$, let $F_{\beta}=\sum_{j=1}^N \beta_j\mathds 1_{A_j^n}$. Assume that Assumption ${\bf A_4}$ is satisfied. If
$$
\frac{1}{\sqrt{nv_n\gamma_n}}\zeta \bigg(\sqrt{\frac{2\bar \phi}{v_n\gamma_n\inf_{\mathscr X} g}}\bigg) \to 0,
$$
then
$$\lim_{n \to \infty}\mathds E \sup_{\beta \in T} | A_n(F_{\beta})-A(F_{\beta})|=0.$$
\end{lem}
\begin{proof}
Let 
$$s_n=\sqrt{\frac{2\bar \phi}{v_n\gamma_n\inf_{\mathscr X} g}},$$
and let $\|\beta\|_{\infty}=\max_{1\leq j \leq N}|\beta_j|$ be the supremum norm of $\beta=(\beta_1, \hdots,\beta_N) \in \mathds R^N$. By definition of $T$, we have, for all $\beta \in T$,
$$\sup_{\mathscr X}|F_{\beta}|=\sup_{\mathscr X} \big| \sum_{j=1}^N\beta_j \mathds 1_{A_j^n}\big|\leq \|\beta\|_{\infty}\leq s_n.$$
In addition, according to Assumption ${\bf A_4}$, we may write, for $\beta_1$ and $\beta_2\in T$,
$$|\phi(F_{\beta_1}(x),y)-\phi(F_{\beta_2}(x),y)|\leq \zeta(s_n)|F_{\beta_1}(x)-F_{\beta_2}(x)|\leq \zeta(s_n)\|\beta_1-\beta_2\|_{\infty}.$$
This shows that the process 
$$\Big(\frac{A_n(F_{\beta})-A(F_{\beta})}{\zeta(s_n)}\Big)_{\beta \in T}$$ 
is subgaussian \citep[e.g.,][Chapter 5]{Ha16} for the distance $d(\beta_1,\beta_2)=\frac{1}{\sqrt n}\|\beta_1-\beta_2\|_{\infty}$. 
Now, let $N(T,d,\varepsilon)$ denote the $\varepsilon$-covering number of $T$ for the distance $d$. Then, by Dudley's inequality \citep[][Corollary 5.25]{Ha16}, one has
\begin{align*}
\mathds E \sup_{\beta \in T} ( A_n(F_{\beta})-A(F_{\beta})) & \leq 12\zeta(s_n) \int_0^{\infty}\sqrt{\log \Big(N(T, \frac{1}{\sqrt n}\|\cdot\|_{\infty},\varepsilon)\Big)}{\rm d}\varepsilon\\
&=12 \zeta(s_n)\cdot \frac{1}{\sqrt n}\int_0^{\infty}\sqrt{\log (N(T, \|\cdot\|_{\infty},\varepsilon))}{\rm d}\varepsilon.
\end{align*}
Let $B_2(0,1)$ denote the unit Euclidean ball in $(\mathds R^N,\|\cdot\|_2)$. Since $T=s_nB_2(0,1)$, we see that
$$\mathds E \sup_{\beta \in T} ( A_n(F_{\beta})-A(F_{\beta}))\leq 12 \zeta(s_n)\cdot \frac{s_n}{\sqrt n}\int_0^{\infty}\sqrt{\log (B_2(0,1), \|\cdot\|_{\infty},\varepsilon)}{\rm d}\varepsilon.$$
But $\|\cdot \|_2\leq \sqrt N \|\cdot\|_{\infty}$, and so
\begin{align*}
\mathds E \sup_{\beta \in T} ( A_n(F_{\beta})-A(F_{\beta}))&\leq 12 \zeta(s_n)\cdot \frac{s_n}{\sqrt n}\int_0^{\infty}\sqrt{\log \Big(B_2(0,1), \frac{1}{\sqrt N}\|\cdot\|_{2},\varepsilon \Big)}{\rm d}\varepsilon\\
&=12 \zeta(s_n)\cdot \frac{s_n}{\sqrt n}\cdot\frac{1}{\sqrt N}\int_0^{\infty}\sqrt{\log(3/\varepsilon)^N}{\rm d}\varepsilon\\
&= 12 \,\frac{s_n\zeta(s_n)}{\sqrt n}\int_0^{\infty}\sqrt{\log(3/\varepsilon)}{\rm d}\varepsilon.
\end{align*}
In the last equality, we used the fact that $N(B_2(0,1),\|\cdot\|_2,\varepsilon)$ equals 1 for $\varepsilon \geq 1$ and is not larger than $(3/\varepsilon)^N$ for $\varepsilon<1$ \citep[e.g.,][Chapter 5]{Ha16}. The same conclusion holds for $\mathds E \sup_{\beta \in T} ( A(F_{\beta})-A_n(F_{\beta}))$, and this proves the result.
\end{proof}

\bibliography{biblio-bc}

\end{document}